\documentclass[11pt,twoside,english]{article}
\usepackage{amsmath}
\usepackage{amsfonts}
\usepackage{mathrsfs}
\usepackage{color}
\usepackage{ amsmath, amsfonts, amssymb, amsthm, amscd}
\usepackage[T1]{fontenc}
\usepackage[english]{babel}
\usepackage[noinfoline]{imsart}

\newtheorem{theorem}{Theorem}
\newtheorem{lemma}{Lemma}
\newtheorem{proposition}{Proposition}

\newtheorem{definition}{Definition}
\newtheorem{corollary}{Corollary}
\newtheorem{remark}{Remark}

\newcommand{\cF}{\ensuremath{\mathcal F}}
\newcommand{\cG}{\ensuremath{\mathcal G}}
\newcommand{\cH}{\ensuremath{\mathcal H}}

\newcommand{\cM}{\ensuremath{\mathcal M}}

\newcommand{\bbE}{{\ensuremath{\mathbb E}} }

\newcommand{\bbP}{{\ensuremath{\mathbb P}} }

\newcommand{\bbR}{{\ensuremath{\mathbb R}} }

\newcommand{\N}{\mathbb{N}}

\newcommand{\be}{\begin{equation}}
\newcommand{\ee}{\end{equation}}
\newcommand{\beq}{\begin{eqnarray}}
\newcommand{\eeq}{\end{eqnarray}}

\newcommand{\1}{{1} \hspace{-0.25 em}{\rm I}}

\newcommand{\R}{\mathbb{R}}

\newcommand{\E}{\mathbb{E}}
\newcommand{\dis}{\displaystyle}

\newcommand{\ced}{\end{proof}}

\newcommand{\red}{}
\newcommand{\blue}{}

\newcommand{\bK}{\bar{K}}
\newcommand{\bY}{\bar{Y}}
\newcommand{\bZ}{\bar{Z}}
\begin{document}
\begin{frontmatter}
\title{Quasilinear Stochastic PDEs with two obstacles: Probabilistic approach}
\date{}
\runtitle{}
\author{\fnms{Laurent}
 \snm{DENIS}\corref{}\ead[label=e1]{ldenis@univ-lemans.fr}}
\thankstext{T1}{The work of the first two authors is supported by the research projects {\it DEFIMATHS} and {\it PANORisk}, r\'egion Pays de la Loire.}
\address{Laboratoire Manceau de Math\'ematiques\\Institut du Risque et de l'Assurance \\Le Mans Universit\'e
\\\printead{e1}}
\author{\fnms{Anis}
 \snm{MATOUSSI}\corref{}\ead[label=e2]{anis.matoussi@univ-lemans.fr}}
\thankstext{t2}{The second author's research is supported by the Chair Risques  \'Emergents ou atypiques en Assurance, under the aegis of Fondation du Risque, a joint initiative by Le Mans Universit\'e, \'Ecole Polytechnique and l'Enreprise MMA. }
\address{Laboratoire Manceau de Math\'ematiques \\ Institut du Risque et de l'Assurance \\Le Mans Universit\'e
  \\\printead{e2}}

\author{\fnms{Jing}
 \snm{ZHANG}\corref{}\ead[label=e3]{zhang\_jing@fudan.edu.cn}}
\thankstext{T3}{The work of the third author is supported by National Natural Science Foundation of China (12031009, 11701404, 11401108) and Key Laboratory of Mathematics for Nonlinear Sciences (Fudan University), Ministry of Education.}
\address{School of Mathematical Sciences \\
Fudan University  \\\printead{e3}}

\runauthor{L. Denis, A. Matoussi and J. Zhang}

\begin{abstract}
We prove an existence and uniqueness result for two-obstacle problem for quasilinear Stochastic PDEs (DOSPDEs for short).
The method is based on the probabilistic interpretation of the solution by using 
the backward doubly stochastic differential equations (BDSDEs for short). 
\end{abstract}


\end{frontmatter}

%
%
%
%
%
%
%
%
%
%
%
%
%

\section{Introduction}
We consider  the following stochastic partial differential equations (SPDEs for short) in $\mathbb{R}^d$:
\small
\begin{equation}\label{SPDE1}
\begin{split}
du_t (x) + \Big[\frac{1}{2} \Delta u_t (x)+ f_t(x,u_t (x),\nabla u_t (x))&+\mbox{div} g_t \left(x,u_t\left(x\right)\nabla u_t\left( x\right) \right)\Big]dt\\
&+h_t(x,u_t(x),\nabla u_t(x))\cdot \overleftarrow{dB}_t = 0\,,
\end{split}
\end{equation}
over the time interval $[0,T]$, with a given final condition $u_T = \Psi$ and $f, g = \big(g_1, \cdots,g_d\big)$,
$h = \big(h_1, \cdots,h_{d^1}\big)$  non-linear random functions. The differential term with $\overleftarrow{dB}_t$
refers to the backward stochastic integral with respect to a $d^1$-dimensional Brownian motion on
$\big(\Omega, \mathcal{F},\mathbb{P}, (B_t)_{t\geq 0} \big)$ (for the backward integral see \cite{Kunita}, Page 111-112). {\red{ We use the backward notation because our approach is fundamentally based on the doubly stochastic framework introduced in the seminal paper by Pardoux and Peng \cite{PardouxPeng94}.    The class of stochastic PDEs as in \eqref{SPDE1}  and their extensions is an important one, since it arises in a number of applications, ranging from asymptotic limits of partial differential equations (PDEs for short) with rapid (mixing) oscillations in time, phase transitions and front propagation in random media with random normal velocities, filtering and stochastic control with partial observations, path-wise stochastic control theory, mathematical finance. The main difficulties with equations like \eqref{SPDE1} are even in the deterministic case, there are no smooth and no explicit solutions in general.\\
The starting point of the theory of SPDEs was with classical solutions in a linear context, wellposedness results having been obtained notably by Pardoux \cite{Pardoux1}, Dawson \cite{Dawson72}, or Krylov and Rozovsk{\u\i}  \cite{KrylovRozovski81}.  In the case when the coefficient $g = 0$, extensions have been obtained later, notably by Pardoux and Peng \cite{PardouxPeng94} (see  also Krylov and Rozovsk{\u\i}  \cite{KrylovRozovski81}, Bally and Matoussi \cite{BM01}) by introducing backward doubly stochastic differential equations (BDSDEs for short), which allowed them to give a nonlinear Feynman-Kac formula for SPDE \eqref{SPDE1} .
The theory of BDSDEs has then been extended in several directions, notably by Matoussi and Scheutzow \cite{MS02} who considered a class of BDSDEs where the nonlinear noise term is given by the more general It\^o-Kunita stochastic integral, thus allowing them to give a probabilistic interpretation of classical and Sobolev solutions of semi--linear parabolic SPDEs driven by Kunita-type martingale fields with specific spatial covariance structure. }}\\
Given two obstacles $\underline{v}$ and $\overline{v}$, our aim in this paper, is to study the two-obstacle problem for SPDE \eqref{SPDE1}, i.e. we want to find a solution of \eqref{SPDE1} which satisfies "$\underline{v}\leq u\leq\overline{v}$". The obstacles should be "regular" in some sense (see Section \ref{hypotheses}). \\
Matoussi and Stoica \cite{MatoussiStoica} have proved an existence and
uniqueness result for the one-obstacle problem for SPDE \eqref{SPDE1}.  The method is based on the probabilistic
interpretation of the solution by using the backward doubly stochastic differential equation.  They have proved that
the solution  is  a pair $\left( u,\nu \right)$ where  $u$ is a predictable
continuous process which takes values in a proper  Sobolev space and
$\nu$ is a random regular measure satisfying minimal Skohorod
condition. In particular  they gave  the regular measure $\nu$ a  probabilistic interpretation in terms of the continuous increasing  process $K$ in the solution $ (Y,Z,K)$  of a  reflected generalized BDSDE. \\
{ \red{The aim of this work is to apply the same approach (as in  \cite{MatoussiStoica})  in the case of two obstacles by introducing  two reflected generalized BDSDEs, allowing  a probabilistic representation of solutions to SPDEs with two obstacles}}.  But, similarly to BSDEs theory, this generalization to the case of two obstacles is not so obvious, and we'll have to impose separability on the obstacles and a kind of Mokobodsky condition (hypothesis {\bf (HO)-(iii)}), see \cite{Cvitan-Kara, HamLepMat97, Hamadene-Hasani,LepeltierMartin} for the BSDEs case. More precisely, we first have to give a sense to the following DOSPDE:
\begin{equation}
\label{DOSPDE0} \left\{\begin{split} &d u (t,x) +\frac{1}{2} \Delta u (t,x)dt  + f(t,x,u_t (x), \nabla u_t (x))dt+\mbox{div} g(t,x,u_t (x), \nabla u_t (x))dt 
\\&\qquad\qquad\qquad\qquad\qquad+h(t,x,u_t (x), \nabla u_t (x))\cdot\overleftarrow{dB}_t +\nu^+ (dt,x)-\nu^- (dt,x)=0,\\
 &\underline{v}(t,x)\leq u(t,x)\leq \overline{v}(t,x),\\
 &\int_0^T\int_{\bbR^d} \left( \tilde{u}(t,x)-\underline{v}(t,x)\right)\nu^+ (dt,dx)=\int_0^T\int_{\bbR^d} \left( \overline{v}(t,x)-\tilde{u}(t,x)\right)\nu^- (dt,dx)=0,\\
 &u_T=\Psi, \end{split}\right. 
\end{equation}
where $\nu^+$ (resp. $\nu^- $) is a measure pushing up (resp. pushing down) the solution when it reaches the lower barrier (resp. upper barrier) and $\tilde{u}$ a quasi-continuous version of the solution. Then, we prove the existence and uniqueness under Lipschitz conditions on the coefficients by using a penalization argument.\\ 
Let us mention that in Denis, Matoussi and Zhang  \cite{DMZ12},  an existence and uniqueness result for the one-obstacle problem of forward quasilinear stochastic PDEs on an open domain in $\mathbb{R}^d$ and driven by an infinite dimensional Brownian motion  is proved. The method is based on analytical technics coming from the parabolic potential theory. The key point was to construct a solution which admits a quasi-continuous version defined outside a polar set and the regular measures which in general are not absolutely continuous w.r.t. the Lebesgue measure, do not charge polar sets. Unfortunately, up to now, we are not able to generalize this analytical approach to the two-obstacle case. {\red Let us explain the difficulties we face in the analytical case: a natural approach consists in considering the solution of the SPDE which is reflected on the lower barrier and penalized on the above barrier:
\begin{equation*}
 \left\{\begin{split} &d u^n (t,x) +\frac{1}{2} \Delta u^n (t,x)dt  + f(t,x,u^n_t (x), \nabla u^n_t (x))dt+\mbox{div} g(t,x,u^n_t (x), \nabla u^n_t (x))dt 
\\&\qquad+h(t,x,u^n_t (x), \nabla u^n_t (x))\cdot\overleftarrow{dB}_t -n(u^n (t,x)-\overline{v}(t,x))^+ +\nu^{+,n} (dt,x)=0,\\
 &\underline{v}(t,x)\leq u^n(t,x),\\
 &\int_0^T\int_{\bbR^d} \left( {u^n}(t,x)-\underline{v}(t,x)\right)\nu^{+,n} (dt,dx)=0,\\
 &u^n_T=\Psi, \end{split}\right. 
\end{equation*}
and to make $n$ tend to $+\infty$. \\
The convergence of the measures $\nu^{+,n}$ and 
$\nu^{-,n}=n(u^n (t,x)-\overline{v}(t,x))^+ dtdx$ is not obvious even if we can control both  the $H^1$-norm of $u^n$ and  the sequence of signed measures $\nu^n =\nu^{+,n}-\nu^{-,n}$  but when passing to the limit, we are  not able to ``separate'' the limit measure in two {\it regular measures} $\nu^+$ and $\nu^-$. Whereas in the probabilistic approach, we succeed thanks to the fact that each regular measure is associated with a continuous increasing process (see \cite{MatoussiStoica}, Theorem 2, assertion (v)) and to pass to the limit we apply a very strong result of convergence for  semimartingales due to Peng and Xu in \cite{PengXu}, known as the stochastic monotonic convergence theorem,  see Lemma \ref{contK+} and the beginning of Section \ref{FundLemma} below.
}\\
The paper is divided as follows: in the second section, we recall the objects coming from the potential theory that we will use and introduce the notion of random extended regular measure. In Section 3, we set the hypotheses and present the main result of this paper. The fourth section is devoted to proving the existence and uniqueness of the solution. To do that we begin with the linear case, and then by Picard iteration we get the result in the nonlinear case. We also establish an It\^o's formula and a comparison theorem. The last section is an Appendix in which we give the proofs of several lemmas. 
\section{Preliminaries}
\label{preliminary}
\subsection{Functional spaces}\label{spaces}
The basic Hilbert space of our framework is $L^2\left( \mathbb{{R}}^d\right) $ and we employ the usual notations for its scalar product and its norm:
$$(u,v):=\int_{\mathbb{R}^d}u(x)v(x)dx\quad\mbox{and}\quad\left\| u\right\|:=\left( \int_{\mathbb{R}^d}u^2(x)dx\right)^{\frac 12}.$$

Our evolution problem will be considered over a fixed time interval $[0,T]$
and the norm for a function $L^2\left( [0,T] \times \mathbb{{R}}^d\right) $ will be denoted by $$\left\| u\right\| _{2,2}=\left(\int_0^T  \int_{\mathbb{R}^d} |u (t,x)|^2 dx dt \right)^{\frac 12}. $$

Another Hilbert space that we use is the first order Sobolev space $H^1(\mathbb{R }^d)= H_0^1(\mathbb{R }^d)$. Its natural scalar product and norm are
$$(u,v) _{H^1(\mathbb{R}^d)}:=(u,v) +(\nabla u, \nabla v)\quad\mbox{and}\quad\left\| u\right\|_{H^1(\mathbb{R}^d)}:=\left(\left\|u\right\|^2+\left\| \nabla u\right\|^2\right)^{\frac 12},$$
where we denote the gradient by $\nabla u (t,x) = (\partial_1 u (t,x),\cdots,\partial_d u (t,x))$.

Of special interest is the subspace $\widetilde{F} \subset L^2( [0,T]; H^1({\mathbb{R}^d})) $ consisting of all functions $u(t,x)$ such that $ t \longmapsto u_t = u(t, \cdot)$ is continuous in $L^2 (\mathbb{R}^d)$. The natural norm on $\widetilde{F}$ is
$$\left\| u\right\| _{T}:=\left( \sup_{ 0 \leq t \leq T}\left\| u_t\right\|^2 + \int_0^T  \|\nabla u_t \|^2 dt \right)^{\frac 12}.$$

The Lebesgue measure on $\mathbb{R}^d$ will be sometimes denoted by $m$. The space of test functions which we employ in the definition of weak solutions  is $ \mathcal{D}_T  := \mathcal{C}^{\infty} \big([0,T]\big) \otimes \mathcal{C}_c^{\infty} \big(\mathbb{R}^d\big)$, where $\mathcal{C}^{\infty} \big([0,T]\big)$ denotes the space of real functions which can be extended as infinite differentiable functions in the neighborhood of $[0,T]$ and $ \mathcal{C}_c^{\infty}\big(\mathbb{R}^d\big)$ is the space of infinite differentiable functions with compact support in $\mathbb{R}^d$.

\subsection{Parabolic potential theory notions and regular measures}\label{measure}

We  present in this subsection the main notions and objects coming from the parabolic potential theory we shall use, for more details we refer to \cite{MatoussiStoica} Section 2. We also introduce what we call {\it extended regular measures}.\\

The operator $\frac{1}{2} \Delta $ is probabilistically interpreted by the Bownian motion in $\mathbb{R}^d$.  We shall view the Brownian motion as a Markov process, $(W_t)_{t \geq 0}$, defined on  the canonical space  $ \Omega' = \mathcal{C }\left([0, \infty ); \mathbb{R}^d \right)$, by $ W_t (\omega) = \omega (t)$, for any $ \omega \in \Omega'$, $t \geq 0$.
 The canonical filtration $ \mathcal{F}_t = \sigma \left( W_s; s \leq t \right)$ is completed by the standard procedure.
We shall also use the backward filtration of the future events $ \mathcal{F}'_t = \sigma \left(W_s; \; \,  s \geq t \right)$ for $t\geq 0$. $\mathbb{P}^0$ is the Wiener measure, which is supported by the set $ \Omega'_0 = \{ \omega \in \Omega', \; \,  w (0) =0 \}$. We also set $ \Pi_0 (\omega) (t) := \omega (t) - \omega (0),\,  t \geq 0$, which defines a map $ \Pi_0 : \Omega' \rightarrow \Omega'_0$. Then  $\Pi := (W_0, \Pi_0 ) \, : \,  \Omega' \rightarrow \mathbb{R}^d \times \Omega'_0$ is a bijection. For each  measure $\mu$ on $\mathbb{R}^d$, the measure 
$\mathbb{P}^{\mu}$ of the Brownian motion started with the initial distribution $\mu$ is given by $$ \mathbb{P}^{\mu} = \Pi^{-1} \left(\mu \otimes \mathbb{P}^0 \right).$$
In particular, for the Lebesgue measure on $\mathbb{R}^d$, which we denote by $ m = dx$, we have
$$ \mathbb{P}^{m} = \Pi^{-1} \left(dx\otimes \mathbb{P}^0 \right),$$
and we'll denote by $\E^m$ the ``expectation'' w.r.t. the measure $ \mathbb{P}^{m}$.\\
It is known that each component $(W^i_t)_{t \geq 0}$ of the Brownian motion, $ i =1,\cdots,d$, is a martingale under any of the measures $\mathbb{P}^{\mu}$. 

 The parabolic operator $\partial _{t}+ \frac{1}{2}\Delta$ can be viewed as the generator of  the time-space Brownian motion, with the state space $\left[ 0,T\right[ \times \mathbb{R}
^{d}$. Its associated
semigroup will be denoted by $(\widetilde{P}_{t}) _{t>0}.$ 
%
It acts as a strongly
continuous semigroup of contractions on the spaces $L^{2}\left( \left[ 0,T
\right[ \times \mathbb{R}^{d}\right): =L^{2}\left( \left[ 0,T\right[
;L^{2}\left( \mathbb{R}^{d}\right) \right) $ and $L^{2}\left( \left[ 0,T
\right[ ;H^{1}\left( \mathbb{R}^{d}\right) \right) .$

The next definition introduces the important notions of {\it quasi-continuity} and {regular potential}:
\begin{definition}
\label{potential} (i) A function $\psi :\left[ 0,T\right] \times \mathbb{R}
^{d}\rightarrow \overline{\mathbb{R}}$ is called quasicontinuous provided
that for each $\varepsilon >0,$ there exists an open set, $D_{\varepsilon
}\subset \left[ 0,T\right] \times \mathbb{R}^{d},$ such that $\psi $ is
finite and continuous on $D_{\varepsilon }^{c}$ and
\begin{equation*}
\mathbb{P}^{m}\left( \left\{ \omega \in \Omega ^{\prime }\;\big|\;\exists \, t\in \left[ 0,T
\right] \ s.t.\left( t,W_{t}\left( \omega \right) \right) \in D_{\varepsilon
}\right\} \right) <\varepsilon .
\end{equation*}

(ii) A function $u:\left[ 0,T\right] \times \mathbb{R}^{d}\rightarrow \left[
0,\infty \right] $ is called a regular potential, provided that its restriction to $[0,T[\times\mathbb{ R}^d$ is
excessive with respect to the  time-space semigroup, it is
quasicontinuous, $u\in \widetilde{F}$ and $\lim_{t\rightarrow T}u_{t}=0$ in $
L^{2}\left( \mathbb{R}^{d}\right) $.
\end{definition}

Observe that if a function $\psi $ is quasicontinuous, then the process $
( \psi _{t}\left( W_{t}\right) ) _{t\in \left[ 0,T\right] }$ is
continuous. 
{\blue The basic properties of the regular potentials are stated in the following theorem (see Theorem 2 in \cite{MatoussiStoica}):}
\begin{theorem}
\label{potentielregulier}
Let $u\in \widetilde{F}.$ Then $u$ has a version which is a regular
potential if and only if there exists a continuous increasing process $
A=\left( A_{t}\right) _{t\in \lbrack 0,T]}$ which is $\big( \mathcal{F}
_{t}\big) _{t\in \lbrack 0,T]}$\textbf{\ -}adapted and such that $
A_{0}=0$, $\mathbb{E}^{m}\left[ A_{T}^{2}\right] <\infty $ and
\begin{equation}
u_{t}(W_{t})=\mathbb{E}^m\left[ A_{T}\,\big|\mathcal{F}_{t}\right]
-A_{t},\;\mathbb{P}^{m}\mathbf{-}a.e.,  \tag{$i$}  \label{caf}
\end{equation}
for any $t\in \left[ 0,T\right] .$ The process $A$ is uniquely determined
by these properties. Moreover, the following relations hold
\begin{equation}
u_{t}\left( W_{t}\right) =A_{T}-A_{t}-\sum_{i=1}^{d}\int_{t}^{T}\partial
_{i}u_{s}\left( W_{s}\right) dW_{s}^{i},\ \mathbb{P}^{m}-a.e.,  \tag{$ii$}
\label{retro}
\end{equation}
\begin{equation}
\left\| u_{t}\right\|^{2}+\int_{t}^{T}\big\|\nabla u_{s}\big\|^2 \,
ds=\mathbb{E}^{m}\left( A_{T}-A_{t}\right) ^{2} , \tag{$iii$}  \label{energ}
\end{equation}
\begin{equation}
\left( u_{0},\mathcal{\varphi }_{0}\right) +\int_{0}^{T} \frac{1}{2} \big( %
\nabla u_{s}, \nabla \mathcal{\varphi }_{s}\big) +\big( u_{s},\partial _{s}%
\mathcal{\varphi }_{s}\big) \, ds=\int_{0}^{T}\int_{\mathbb{R}^{d}}%
\mathcal{\varphi }\left( s,x\right) \nu \left( ds,dx\right) ,  \tag{$iv$}
\label{faible}
\end{equation}%
for any test function $\mathcal{\varphi \in D}_T,$ where $\nu $ is the
measure defined by%
\begin{equation}
\mathcal{\nu }\left( \mathcal{\varphi }\right) =\mathbb{E}^{m}\int_{0}^{T}\mathcal{%
\varphi }\left( t,W_{t}\right) dA_{t},\ \mathcal{\varphi \in C}_{c}\left( %
\left[ 0,T\right] \times \mathbb{R}^{d}\right) , \tag{$v$}  \label{mesure}
\end{equation}
and $\mathcal{C}_c ([0,T]\times \R^d)$ is the set of  continuous functions on $[0,T]\times\R^d$ with compact support.
\end{theorem} 

We now introduce the class of measures which intervene in the notion of
solution to the one-obstacle problem.

\begin{definition}
A nonnegative Radon measure $\nu $ defined on $\left[ 0,T\right] \times
\mathbb{R}^{d}$ is called regular provided that there exists a regular
potential $u$ such that relation (iv) from the above theorem is
satisfied.
\end{definition}
{\blue We denote by $\cM([0,T]\times\bbR^d)$ the collection of all regular measures on $[0,T]\times\bbR^d$ and by $\mathcal{A}_2$ the set of continuous additive functionals $A$ associated to regular measures by relation {\it (v)} above.}

As a consequence of the preceding theorem, we see that the regular measures are
always represented as in relation {\it (v)} of the theorem, with a certain
increasing process. We also note the following properties of a regular
measure, with the notations from the theorem.

\begin{enumerate}
\item A set $B\in \mathcal{B}\left( \left[ 0,T\right] \times \mathbb{R}
^{d}\right) $ satisfies the relation $\nu \left( B\right) =0$ if and only if
$\int_{0}^{T}{\bf 1}_{B}\left( t,W_{t}\right) dA_{t}=0,\ \mathbb{P}^{m}-a.e..$

\item If a set $B\in \mathcal{B}\left( \left] 0,T\right[ \times \mathbb{R}
^{d}\right) $ is polar, in the sense that
\begin{equation*}
\mathbb{P}^m\left( \left\{ \omega \in \Omega ^{\prime }\;\big|\;\exists t\in \left[ 0,T\right]
,\left( t,W_{t}\left( \omega \right) \right) \in B\right\} \right) =0,
\end{equation*}
then $\nu \left( B\right) =0.$

\item If $\psi ^{1},\psi ^{2}:\left[ 0,T\right] \times \mathbb{R}
^{d}\rightarrow \overline{\mathbb{R}}$ are Borel measurable and such that $
\psi ^{1} (t,x)\geq \psi^{2} (t,x),\; dt\otimes dx-a.e.,$   and  the processes $\left( \psi _{t}^{i}\left( W_{t}\right) \right)
_{t\in \left[ 0,T\right] },i=1,2$, are a.s. continuous, then one has $\nu
\left( \psi ^{1} <  \psi ^{2}\right) =0.$
\end{enumerate}
In the case of two obstacles we are obliged to consider a wider class of measures, that's why we introduce the following definition:
\begin{definition} We denote by $\mathcal{A}_1$ the set of continuous increasing processes $(A_t )_{t\in [0,T]}$ with $A_0=0$ and $\E^m [A_T ]<+\infty$ which are  uniform limit of a sequence of elements in $\mathcal{A}_2$   in the  sense that there exists a sequence $(A^n)$ in $\mathcal{A}_2$ such that for all $t\in [0,T]$:
$$\lim_{n\rightarrow +\infty}\sup_{t\in [0,T]}|A^n_t - A_{t}|=0,\ \makebox{ a.e.}.$$
\end{definition}
Let us remark that as a consequence of this definition, any element in $\mathcal{A}_{1}$ is an additive functional.\\
Naturally such a process is associated with a measure:
\begin{proposition}\label{DefMeasure} Let $A\in\mathcal{A}_1$, then there exists a unique Radon measure $\nu$ on $[0,T]\times \R^d$ such that
\begin{equation}\label{FormuleMeasure}\forall \varphi \in \mathcal{C}_c ([0,T]\times \R^d),\  \nu (\varphi )=\E^m \int_0^T \varphi (t,W_t )dA_t .
\end{equation}
Moreover, $\nu$ does not charge polar sets. We shall call such  measure an {\rm extended regular measure}.
\end{proposition}
\begin{proof}
As a consequence of  Daniell's theorem, it is clear that the relation above defines a unique Radon measure on $[0,T]\times \R^d$ and that by uniqueness for any Borel set $B\subset [0,T]\times \R^d$ we have
$$\nu (B)=\E^m \int_0^T{\bf 1}_B (t,W_t )dA_t,$$
ensuring that $\nu$ does not charge polar sets.
\end{proof}

\subsection{The probabilistic interpretation of the divergence term}

Let  $f$ and  $|g|$ belong to $ L^2 \left([0,T] \times \mathbb{R}^d \right)$, $\Psi$ be in $L^2 (\R^d )$ and $ u \in \widetilde{F}$ be the solution of the deterministic  equation 
\begin{equation}
\label{PDE1} \left\{\begin{split}& \partial_t u (t,x) +
 \;  \frac{1}{2} \Delta u (t,x)  + f(t,x)+ \mbox{div} g(t,x) =0,\\
 &u_T=\Psi. \end{split}\right.
\end{equation}
 Let us denote by
\begin{equation}
\label{representation:divg}
\int_s^t g_r * dW_r = \sum_{i=1}^d  \left( \int_s^t g_i (r, W_r) dW_r^i + \int_s^t g_i (r, W_r)\overleftarrow{dW_r^i} \right).
\end{equation}
Then one has the following representation (see Theorem 3.2 in \cite{Stoica}):
\begin{theorem}
The following relation holds $\mathbb{P}^m\textbf{-}a.e.$ for any $0 \leq s \leq t \leq T$,
\begin{equation}
\label{divergence:term}
u_t (W_t) - u_s(W_s) = \sum_{i=1}^d  \int_s^t \partial_i  u_r (W_r) dW_r^i   - \int_s^t f_r (W_r) dr + \frac{1}{2}\int_s^t g_r * dW_r\,.
\end{equation}
\end{theorem}
{ \begin{remark} If $g$ is regular with respect to the space variable, then (see \cite{Stoica})
$${\int_s^t g_r * dW_r=-2\int_s^t \makebox{\rm div} g(r,W_r )\, dr.}$$
Moreover, since $W$ is a reversible Markov process w.r.t. to the invariant measure $m$, by the time change $u=T-r$, 
it appears that, under $\bbP^m$,  the process $(\int_0^t g_i (r, W_r) \overleftarrow{dW_r^i})_{t\in [0,T]}$ has the same "law"  as 
$(\int_{T-t}^T g_i (T-u, W_u) dW^i_u)_{t\in [0,T]}$, 
hence for example $\E^m [\int_s^t g_r * dW_r ]=0$ and the process $(\int_0^t g_r * dW_r )_{t\in [0,T]}$ satisfies the Burkholder-Davis-Gundy inequality.
\end{remark}}

\subsection{The doubly stochastic framework}\label{DoublySto}

Let $ B:= (B_t)_{t\geq 0}$ be a standard
$d^1$-dimensional Brownian motion  on a probability
space  $ \big( \Omega,\mathcal{F}^B, \mathbb{P }\, \big)$.  
Over the time interval $[0,T]$  we define the backward filtration $\left(\mathcal{F}_{t,T}^{B}\right)_{ t\in [0,T]}$ where $\mathcal{F}^B_{t,T}$ is the completion in $\mathcal{F}^B$ of $\sigma (B_{r}-B_{t}; t\leq r\leq T)$.\\
We denote by $\cH_T$ the space of $H^1(\mathbb{R}^d)$-valued
$\mathcal{F}^B_{t,T}$-adapted processes   $(u_t)_{0\leq t \leq T}$ such that the trajectories $ t \rightarrow u_t $ are in $\widetilde{F}$ a.s. and
$$\mathbb{E}\sup_{t\in[0,T]}\| u_t\|^2+\mathbb{E}\int_0^T\|\nabla u_t\|^2dt<+\infty.$$

 We now remind  the quasicontinuity result of the solution of the linear equation, i.e. when  $f, g, h$ do not depend on $u$ and $ \nabla u$.  To this end we first  extend the doubly stochastic It\^o's formula to our framework. We start by recalling  the following result from \cite{DenisStoica} and \cite{MatoussiStoica}:
\begin{theorem}
\label{FK}
Let $f\in L^2 \left(\Omega\times[0,T] \times \mathbb{R}^d;\mathbb{R}\right)$, $g\in L^2 \left(\Omega\times[0,T] \times \mathbb{R}^d;\mathbb{R}^d \right)$ and $h\in L^2 (\Omega\times[0,T]\times\R^d ;\R^{d^1})$ be predictable processes w.r.t. the backward filtration $(\mathcal{F}^B_{t,T})_{t\in [0,T]}$ and $\Psi\in L^2 (\R^d)$.
Let $u\in \mathcal{H}_{T}$ be the unique solution of the equation
\begin{equation*}\left\{\begin{split}
&du_{t}+\frac{1}{2} \Delta u_t dt + \big(f_t  + div g_{t} \big) dt  +
h_t\cdot\overleftarrow{dB}_t = 0,\\
&u_T=\Psi.
\end{split}\right.\end{equation*}
 
Then, for any $0\leq s\leq t\leq T$, one has the following stochastic representation, $\mathbb{P}^{m}$-a.e.,
\begin{equation}
\label{Ito:u}
u\left( t,W_{t}\right) -u\left( s,W_{s}\right)
=\sum\limits_{i}\int_{s}^{t}\partial _{i}u\left( r,W_{r}\right)
dW_{r}^{i}-\int_{s}^{t}f_r\left(W_{r}\right) dr + \frac{1}{2}
\int_{s}^{t}g_r*dW_r -\int_{s}^{t}h_r\left( W_{r}\right) \cdot \overleftarrow{dB}_r.
\end{equation}

\end{theorem}
We remark that $ \mathcal{F}_T$ and $\mathcal{ F}^B_{0,T}$ are independent under $\mathbb{P}^m\otimes\mathbb{P}$ which is the product measure defined on $\Omega'\otimes\Omega$ and therefore in the above formula the stochastic integrals with respect to $ dW_t$ and $\overleftarrow{dW}_t$ act independently of $\mathcal{ F}_{0,T}^B$  and similarly the integral with respect to
$\overleftarrow{dB}_t$ acts independently of $ \mathcal{F}_T.$ \\[0.2cm]
In particular, the process $ \left(u_t (W_t)\right)_{t\in[0,T]} $ admits a continuous version which we usually denote by $Y= \left(Y_t\right)_{t \in [0,T]}$ and we introduce the notation $ Z_t := \nabla u_t (W_t)$. As  a consequence of this theorem  we have the following result:

\begin{corollary}
\label{FK2}
Under the hypotheses of the preceding theorem, one has the following stochastic representation for $ u^2$, $\mathbb{P}\otimes\mathbb{P}^m\textbf{-}a.e.$, for any $0\leq t\leq T$,
\begin{equation}
\label{Ito:2}
\begin{split}
u_t^2\left(W_{t}\right) &-\Psi^2\big(W_{T}\big) = 2 \int_t^T \big[ \big(u_s f_s\big) (W_s) - \frac{1}{2}|\nabla u_s |^2 (W_s) -
 \langle \nabla u_s, g_s \rangle (W_s) + \frac{1}{2}|h_s|^2 (W_s) \, \big] \, ds \\&-
\int_{t}^{T} \big( u_s g_s \big)(W_s) *dW_s - 2 \sum\limits_{i}\int_{t}^{T} \big(u_s \partial _{i}u_s \big) \left(W_{s}\right)
dW_{s}^{i}  + 2 \int_{t}^{T} \big(u_s h_s \big) \left(W_{s}\right) \cdot \overleftarrow{dB}_s.
\end{split}
\end{equation}
Moreover, one has the estimate
\begin{equation}
\label{estimationYZ}
\mathbb{E} \sup_{t \leq s \leq T} \|u_s\|^2 + \mathbb{E }
\int_t^T \|\nabla u_s \|^2 \, ds \leq c  \, \left[ \|\Psi \|^2 + \mathbb{ E} \int_t^T \Big[ \|f_s\|^2 + \|g_s\|^2 + \|h_s\|^2 \, \Big]\, ds  \right],
\end{equation}
for any $ t\in [0,T]$.
\end{corollary}
\begin{remark}
With the notation introduced above one can rewrite relation \eqref{Ito:2} as
\begin{equation}
\label{Ito:BSDE}
\begin{split}
|Y_t|^2   + \int_t^T |Z_r|^2 dr  =&\, |Y_T|^2 +  2 \int_t^T  Y_r f_r (W_r) dr - 2 \int_t^T \langle Z_r, g_r (W_r) \rangle  \, dr - \int_t^{T} Y_r g_r (W_r) *dW_r \\
& -  2 \sum\limits_{i}\int_{t}^{T} Y_r Z_{i,r}
dW_{r}^{i}  + 2 \int_{t}^{T} Y_r h_r (W_r)  \cdot \overleftarrow{dB}_r + \int_t^T |h _r |^2 (W_r) dr .
\end{split}
\end{equation}
\end{remark}
In the deterministic case, it was proven in \cite{Stoica} that the solution of a quasilinear equation has a quasicontinuous version. The same property holds for the solution of an SPDE  (see Proposition 1 in \cite{MatoussiStoica}):
\begin{proposition}
\label{quasicontinuite:EDPS}
Under the hypotheses of Theorem \ref{FK}, there exists  a function $ \tilde{u} \, : \, [0,T]\times  \Omega \times \mathbb{R}^d \rightarrow \mathbb{R}$
  which is a  quasicontinuous version of $u$, in the sense that  for each $\epsilon >0,$ there exists a predictable
  random set $D_{\omega}^{\epsilon} \subset
[0,T] \times \Omega \times \mathbb{R}^d $
such that $\mathbb{P}$\textbf{-}a.s.  the section $ D_{\omega}^{\epsilon}$ is open and $\tilde{u} \left(\cdot, \omega, \cdot \right) $ is continuous on  its  complement  $\left(D_{\omega}^{\epsilon}\right)^c$ and
$$
\mathbb{P}\otimes \mathbb{P}^m \, \left( (\omega, \omega') \; \big| \; \exists t \in [0,T]\; s.t. \;
\left(t, \omega, W_t (\omega') \right) \in
D_{\omega}^{\epsilon} \, \right) \leq \epsilon .
$$
In particular, the process  $\big(\tilde{u}_t (W_t) \big)_{t \in[0,T]}$ has continuous trajectories,
$\mathbb{P}\otimes \mathbb{P}^m\textbf{-}a.e..$
\end{proposition}
The measures intervening in our equations to force the solution of the SPDE to stay between obstacles are random, so we need to introduce the notion of 
a random
regular measure:
\begin{definition}
\label{random:regular:measure}
We say that $ u \in \mathcal{H}_T$ is a random regular potential provided that $ u(\cdot,\omega,\cdot)$ has a version which is regular potential, $\mathbb{P}(d\omega)\textbf{-} a.s.$. The random variable $ \nu\, :\, \Omega \longrightarrow \mathcal{M} \left([0,T] \times \mathbb{R}^d \right) $ with values in the set of regular measures on $ [0,T] \times \mathbb{R}^d$ is called a regular random measure, provided that there exits a random regular potential $u$ such that the  measure $ \nu(\omega) (dt,dx)$ is associated to the regular potential $ u (\cdot, \omega, \cdot)$, $\mathbb{P}(d\omega)\textbf{-} a.s.$.
\end{definition}
The relation between a random measure and its associated random regular potential is described by the following proposition (see \cite{MatoussiStoica}, Proposition 2):
\begin{proposition}
\label{quasicontinuite:bis}
Let $u$ be  a random regular potential and $\nu $ be the associated random regular measure. Let  $\tilde{u}$ be the excessive version of $u$, i.e. 
$\tilde{u}\left( \cdot ,\omega ,\cdot \right) $ is a.s. a $\big(\widetilde{P}_{t}\big)_{t>0}$-excessive function which coincides with 
$u\left( \cdot ,\omega ,\cdot \right)$, $dtdx$-a.e.. Then we have the following properties:\\[0.2cm]
(i) For each $\varepsilon >0,$ there exists a $\left( \mathcal{F}_{t,T}^{B}\right) _{t\in \left[ 0,T\right] }$-predictable random set 
$D_{\omega}^{\varepsilon }\subset \left[ 0,T\right] \times \Omega \times \mathbb{R}^{d}$ 
such that $\bbP$-a.s. the section $D_{\omega }^{\varepsilon }$ is open and 
$\tilde{u}\left( \cdot ,\omega ,\cdot \right) $ is continuous on its
complement $\left( D_{\omega }^{\varepsilon }\right) ^{c}$ and
\begin{equation*}
\mathbb{P}\otimes \mathbb{P}^{m}\left( \left( \omega ,\omega ^{\prime }\right) |\,\exists
t\in \left[ 0,T\right] \ \ s.t.\left( t,\omega ,W_{t}\left( \omega' \right)
\right) \in D_{\omega }^{\varepsilon }\right) \leq \varepsilon .
\end{equation*}
In particular  the process  $\big(\tilde{u}_t (W_t) \big)_{t \in[0,T]}$ has continuous trajectories,
$\mathbb{P}\otimes \mathbb{P}^m\textbf{-}a.e..$\\[0.2cm]
(ii) There exists a continuous increasing process $A:=\left( A_{t}\right)
_{t\in \left[ 0,T\right] }$ defined on $\Omega \times \Omega ^{\prime }$
such that $A_{s}-A_{t}$ is measurable with respect to the $\mathbb{P}\otimes \mathbb{P}^{m}$-completion of 
$\mathcal{F}_{t,T}^{B}\vee \sigma \left( W_{r}, r\in \left[t,s\right] \right) $, for any $ 0 \leq t \leq s \leq T$, 
and such that the following relations are fulfilled almost surely, with any $\mathcal{\varphi }\in \mathcal{D}_T$ and $t\in \left[ 0,T\right],$
\begin{equation*}
\begin{split}
&(a) \quad \left( u_{t},\mathcal{\varphi }_{t}\right) +\int_{t}^{T}\left( \frac{1}{2}
\left( \nabla u_{s},\nabla \mathcal{\varphi }_{s}\right) +\left( u_{s},\partial _{s}
\mathcal{\varphi }_{s}\right) \right) ds=\int_{t}^{T}\int_{\mathbb{R}^{d}}
\mathcal{\varphi }\left( s,x\right) \nu \left( ds,dx\right) ,\\
& (b)  \quad u_{t}(W_{t})=\mathbb{E}^m\left[ A_{T}\,\big|\mathcal{F}_{t}\right] -A_{t}\,,\\
& (c) \quad u_{t}\left( W_{t}\right) =A_{T}-A_{t}-\sum_{i=1}^{d}\int_{t}^{T}\partial
_{i}u_{s}\left( W_{s}\right) dW_{s}^{i},\\
& (d) \quad \left\| u_{t}\right\| ^{2}+ \int_{t}^{T} \big\| \nabla u_{s}\big\| ^2
\, ds=\mathbb{E}^{m}\left( A_{T}-A_{t}\right) ^{2},\\
& (e) \quad \mathcal{\nu }\left( \mathcal{\varphi }\right) =\mathbb{E}^{m}\int_{0}^{T}\mathcal{
\varphi }\left( t,W_{t}\right) dA_{t}\,.
\end{split}
\end{equation*}

\end{proposition}

We remark that, taking the expectation in relation  (ii-d)  above
proposition, one gets $$\mathbb{E}\mathbb{E}^{m}\left[ A_{T}^{2}\right] =\mathbb{E}\left[ \left\|
u_{0}\right\| ^{2}+\int_{0}^{T}\big\|\nabla u_{t}\big\|^2 \,  dt\right] .$$
In a natural way, we define the notion of {\it random extended regular measure} as following:
\begin{definition} A random measure $\nu$ defined on $(\Omega , \mathcal{F}^B)$ and taking values in the set of Radon measures on $[0,T]\times \R^d$  is a {\rm random extended  regular measure} if there exists an increasing process $A$ such that 
$$\forall \varphi\in \mathcal{C}_b ([0,T]\times \R^d ),\, \nu (\varphi )=\E^m\left [\int_0^T \varphi (t,W_t )dA_t \right],$$
where $A$ satisfies the following hypotheses:
\begin{enumerate}
\item $A_0 =0$ and  $\,\E\E^m [A_T]<+\infty$.
\item There exists a sequence of processes $(A^n)$ associated with random regular measures as in Proposition \ref{quasicontinuite:bis} such that
$$\lim_{n\rightarrow +\infty} \sup_{t\in [0,T]}|A_t -A^n_t|=0,\ \mathbb{P}\otimes\mathbb{P}^m-a.e..$$
\end{enumerate}
\end{definition}

\section{Hypotheses and main result}
We consider the following quasilinear parabolic SPDE with two obstacles that, for the moment, we formally write as
\begin{equation}
\label{DOSPDE} \left\{\begin{split} &d u (t,x) +\frac{1}{2} \Delta u (t,x)dt  + f(t,x,u_t (x), \nabla u_t (x))dt+\mbox{div} g(t,x,u_t (x), \nabla u_t (x))dt 
\\&\qquad\qquad\qquad\qquad\qquad+h(t,x,u_t (x), \nabla u_t (x))\cdot\overleftarrow{dB}_t +\nu^+ (dt,x)-\nu^- (dt,x )=0,\\
 &\underline{v}(t,x)\leq u(t,x)\leq \overline{v}(t,x),\\
 &\int_0^T\int_{\bbR^d} \left( \tilde{u}(t,x)-\underline{v}(t,x)\right)\nu^+ (dt,dx)=\int_0^T\int_{\bbR^d} \left( \overline{v}(t,x)-\tilde{u}(t,x)\right)\nu^- (dt,dx)=0,\\
 &u_T=\Psi. \end{split}\right. 
\end{equation}
\begin{remark} As explained in \cite{MatoussiStoica} (Remark 1, p. 1157), we can consider the more general case where the operator 
$\frac12\Delta$ is replaced by a strictly elliptic operator in divergence form $L:=\sum_{ij}\partial _{i}a^{ij}\partial _{j}$, where $a:=\left( a^{ij}\right) :\mathbb{R}
^{d}\rightarrow \mathbb{R}^{d}\times \mathbb{R}^{d}$ is symmetric and measurable. 
\end{remark}

\subsection{Hypotheses}\label{hypotheses}In the remainder of this paper we assume that the final condition $\Psi$ is a given function in $L^2 (\mathbb{R}^d)$ and the functions appearing in \eqref{DOSPDE}
\begin{eqnarray*}
f &:& [0,T]\times \Omega \times  \mathbb{R}^d \times \bbR \times \bbR^{d}
\rightarrow \bbR \;, \\
g&:=& (g_1,...,g_d) \, : \, [0,T]\times \Omega \times \mathbb{R}^d \times
\bbR \times \bbR^{d} \rightarrow \bbR^{d}\;, \\
h&:=& (h_1,...,h_{d^1}) \, : \,  [0,T]\times \Omega \times  \mathbb{R}^d \times \bbR \times \bbR^{d}
\rightarrow \bbR^{d^1} \; 
\end{eqnarray*}
are random functions predictable with respect to the backward filtration  $\left(\mathcal{F}_{t,T}^{B}\right)_{ t\in [0,T]}$. We set
\begin{equation*}
\begin{split}
 &f ( \cdot
,\cdot,\cdot, 0,0)=:f^0,
\quad  g( \cdot,\cdot,\cdot ,0,0) =:g^0 = (g_1^0,...,g_d^0), \quad   h ( \cdot,\cdot, \cdot ,0,0)=:h^0=(h_1^0,...,h_{d^1}^0)\\
\end{split}
\end{equation*}
and assume the following hypotheses: \\[0.2cm]
 \textbf{Assumption (H)}: There exist  non-negative
 constants $C,\, \alpha, \,\beta $ such that
\begin{enumerate}
\item[\textbf{(i)}]
$ |f(t,\omega,x,y,z) -f( t,\omega, x,y',z') | \leq \, C\big(|y-y'| + |z-z'| \big);$
\item[\textbf{(ii)}] 
$ \Big(\sum_{i=1}^{d}| g_{i}(t,\omega,x,y,z) -g_i( t,\omega,x,y',z')|^2\Big)^{\frac{1}{2}} \leq  C \, | y-y'|+ \, \alpha \, |z-z'|;$
\item[\textbf{(iii)}]
$\Big(\sum_{j=1}^{d^1}| h_{j}(t,\omega,x,y,z) -h_j( t,\omega,x,y',z')|^2\Big)^{\frac{1}{2}}\leq C \, | y-y'|+\, \beta \, |z-z'|;$
\item[\textbf{(iv)}]  the contraction property: $ \alpha  +\dis \frac{\beta ^{2}}{2}  < \frac{1}{2}\, $.
\end{enumerate}
\begin{remark} 
In the case when the operator  $\frac12\Delta$ is replaced by  a strictly elliptic operator in divergence form $L:=\sum_{ij}\partial _{i}a^{ij}\partial _{j}$ with  $a:=\left( a^{ij}\right) :\mathbb{R}
^{d}\rightarrow \mathbb{R}^{d}\times \mathbb{R}^{d}$ symmetric and measurable and such that%
$$
\lambda \left\vert \xi \right\vert ^{2}\leq \sum_{ij}a^{ij}\left( x\right)
\xi ^{i}\xi ^{j}\leq \Lambda \left\vert \xi \right\vert ^{2}.
$$
The contraction property becomes :
$ \alpha  +\dis \frac{\beta ^{2}}{2}  < \lambda\, $ (see \cite{MatoussiStoica}, Remark 1,  p. 1157).

\end{remark}
\textbf{Assumption (HD2)}:
$$ \mathbb{E} \left(  \left\| f^0\right\|_{2,2}^2+\left\|
g^0\right\| _{2,2}^2+\left\| h^0\right\| _{2,2}^2\right) <+\infty\,.
 $$
 \\[0.2cm]
{{\textbf{Assumption (HO)} : }The two obstacles $\underline{v} (t,\omega, x)$ and $\overline{v} (t,\omega, x)$ are predictable random functions  with respect to the backward filtration  $\left(\mathcal{F}^B_{t,T}\right)_{t\in[0,T]}$. We also assume  that
\begin{enumerate} 
\item[\textbf{(i)}] $t\rightarrow\underline{v} (t,\omega, x)$ and $t\rightarrow\overline{v} (t,\omega, x)$ are $\mathbb{P} \otimes m-$a.e. continuous on $[0,T]$.
\item[\textbf{(ii)}]  $\underline{v} (T, \cdot) \leq \Psi (\cdot)\leq \overline{v}(T,\cdot)$.
\item[\textbf{(iii)}] There exist  $\tilde f\in L^2 \left(\Omega\times[0,T] \times \mathbb{R}^d;\mathbb{R}\right)$, $\tilde g\in   L^2 \left(\Omega\times[0,T] \times \mathbb{R}^d;\mathbb{R}^d \right)$, $\tilde h \in L^2 (\Omega\times[0,T]\times\R^d ;\R^{d^1})$,  predictable processes w.r.t. the backward filtration $\left(\mathcal{F}^B_{t,T}\right)_{t\in [0,T]}$ and $\tilde \Psi\in L^2 (\R^d)$ such that if we denote by $z$ the solution of the following linear SPDE
\begin{equation}\label{obstacle}
\left\{\begin{split}
&dz_t+\frac{1}{2}\Delta z_tdt+\tilde{f}_tdt+{\rm div} \tilde g_tdt+ \tilde h_t\cdot\overleftarrow{dB}_t=0,\\& z_T=\tilde \Psi,
\end{split}\right.
\end{equation} 
then $\underline{v}_t\leq z_t\leq\overline{v}_t$ a.e. $\forall t\in [0,T]$. 
\item[\textbf{(iv)}] {\rm Strict separability of the obstacles:} there exists  a positive constant $\kappa$   such that 
$\underline{v}_t- z_t\leq-\kappa<0\leq\overline{v}_t-z_t$. 
\end{enumerate}
\begin{remark}
The condition $\textbf{(iv)}$ is similar to the so-called Mokobodski condition used in stochastic Dynkin games.
\end{remark}
%
\begin{remark}\label{IntergrabilitiObstacles} By Theorem 8 in \cite{DenisStoica} we have the existence and uniqueness of $z$. Moreover, we know that $\E\E^m [\sup_{t\in[0,T]}|z_t (W_t )|^2]<+\infty$. Hence, hypothesis ${\bf(iii)}$ of {\textbf{(HO)}} ensures that $$\mathbb{E}\mathbb{E}^m\left[\sup_{t\in[0,T]}(\underline{v}^+(t,\cdot,W_t))^2\right]<+\infty\makebox{ and }\,\mathbb{E}\mathbb{E}^m\left[\sup_{t\in[0,T]}(\overline{v}^-(t,\cdot,W_t))^2\right]<+\infty.$$
\end{remark}}
\subsection{The weak solution for the two-obstacle problem}
We now precise the definition  of the solution of our obstacle problem. We recall that the {\red datum} satisfy the hypotheses of Section \ref{hypotheses}.
\begin{definition}
\label{o-spde} We say that a triplet $(u,\nu^+ ,\nu^- )$ is a weak solution of the two-obstacle problem for SPDE
\eqref{SPDE1} associated to $(\Psi,f,g,h,\underline{v},\overline{v})$, if\\[0.2cm]
(i) $ u \in \mathcal{H}_T $,  $\underline{v}(t,x)\leq u (t,x)\leq \overline{v}(t,x)$ 
$d\mathbb{P}\otimes dt\otimes dx-a.e.$ and $u(T,x)=\Psi(x)$,  $d\mathbb{P}\otimes dx-a.e.;$\\[0.2cm]
(ii)  $\nu^+ $ and $\nu^-$ are random extended regular
measures on {\red $[0,T] \times \mathbb{R}^d$};\\[0.2cm]
(iii)  for any $\varphi \in \mathcal{D}_T$ and $t \in [0,T]$, the following relation holds almost surely, 
\begin{equation}
\label{weak:RSPDE}
\begin{split}
& \int_{t}^{T } \left[\left( u_{s},\partial_{s}\varphi_s \right) +
\frac{1}{2} \left( \nabla u_{s},\nabla \varphi_s \right) \,  \right] \,  ds - \big(\Psi,
\varphi_T \big) +
\big( u_t, \varphi_t \big) \\ =& \int_{t}^{T} \left[ \, \left(f_s \big(u_{s},\nabla u_s \big), \varphi_s \right)  -
  \left( g_s\big(u_{s},\nabla u_s \big), \nabla \varphi_s\right)  \, \right] \, ds \\
 & + \int_{t}^{T} \left( h_s\left( u_{s},\nabla u_s\right) ,\varphi_s  \right) \cdot \overleftarrow{dB}_{s}
 + \int_{t}^{T}\int_{\mathbb{R}^d}\varphi_s(x)\, (\nu^+ -\nu^-) (ds,dx);\\
\end{split}
\end{equation}
(iv) $u$ admits a quasi-continuous version, $\tilde{u}$, and we have%
\begin{equation*}
\int_{0}^{T}\int_{\mathbb{R}^{d}}\left( \tilde{u}_{s}\left( x\right)
-\underline{v}_{s}\left( x\right) \right) \nu^+ \left(ds, dx\right) =\int_{0}^{T}\int_{\mathbb{R}^{d}}\left( 
\overline{v}_{s} \left( x\right)
-\tilde{u}_s \left( x\right) \right) \nu^- \left(ds,dx\right) =0, \; a.s..
\end{equation*}
\end{definition}



\subsection{The main theorem}
Here is the main result of our paper : 
\begin{theorem}
\label{maintheorem}
Suppose that Assumptions \textbf{(H)},  \textbf{(HD2)} and \textbf{(HO)} hold.
Then  there exists a unique weak solution $(u,\nu^+ ,\nu^-)$ of the two-obstacle problem for SPDE  \eqref{SPDE1} associated to $(\Psi,f,g,h,\underline{v},\overline{v})$.
\\Moreover, the quadruple of processes $\left(Y_t, Z_t, K_t^+,K_t^- \right)_{t \in [0,T]}$, the unique solution of the  following doubly reflected backward doubly stochastic differential equation (in short DRBDSDE) :
\begin{equation}\label{DRBDSDE}
\begin{split}
Y_{t}& = \Psi\left( W_T\right) +\int_{t}^{T}f_s
\left(W_{s},Y_s,Z_s\right) ds -\frac{1}{2}\int_{t}^{T}g_s\left(W_{s},Y_s,Z_s\right) *dW_s  + \int_{t}^{T}h_s
\left(W_{s},Y_s,Z_s\right)\cdot  \overleftarrow{dB}_s \\&\qquad -\sum\limits_{i}\int_{t}^{T}Z_{i,s}dW_{s}^{i}+ K_T^+ - K_t ^+-K_T^-+K_t^- \,
\end{split}
\end{equation}
with $L_t\leq Y_t\leq U_t,\; \forall \, t \in [0,T] $,  $\left(K_t^+\right)_{t \in[0,T]}$ and $\left(K_t^-\right)_{t \in[0,T]}$ being increasing continuous processes and    \begin{equation}
\label{reflection:minimum}
\displaystyle \int_{0}^{T}(Y_t-L_t)dK^+_t = \int_0^T(U_t-Y_t) dK_t^-= 0
\end{equation}
is given by: $Y_t = u(t,W_t)$, $ Z_t =\nabla u(t,W_t)$, $L_t=\underline{v}(t,W_t)$, $U_t=\overline{v}(t,W_t)$ and for any $\mathcal{\varphi }\in \mathcal{D}_T$, $$\mathcal{\nu^+ }\left( \mathcal{\varphi }\right) =\mathbb{E}^{m}\int_{0}^{T}\mathcal{%
\varphi }\left( t,W_{t}\right) dK^+_{t}\ \makebox{ and }\ \mathcal{\nu^- }\left( \mathcal{\varphi }\right) =\mathbb{E}^{m}\int_{0}^{T}\mathcal{%
\varphi }\left( t,W_{t}\right) dK^-_{t}.$$
\end{theorem}
\section{Proof of Theorem \ref{maintheorem}}

In order to solve the problem, we will use the backward stochastic differential equation  technics. 
We shall begin with the linear case whose proof is based on the penalization procedure and then use a fixed point argument. Since we are  first going to consider the solution of our SPDE reflected on the lower obstacle and penalized on the upper obstacle, we recall the result in the one-obstacle case.
\subsection{The probabilistic interpretation of the solution of the one-obstacle problem}
In \cite{MatoussiStoica}, the one-obstacle problem was studied, it corresponds to the case of two-obstacle problem by taking 
$\overline{v}=+\infty$ and $\nu =\nu^+$:
\begin{equation}
\begin{split}
\label{OSPDE1}
& (i') \; \;    u \geq v , \quad d\mathbb{P}\otimes dt\otimes dx - \mbox{a.e.},  \\
& (ii')\;\; du_t (x) + \big[
 \;  \frac{1}{2} \Delta u_t (x)  + f_t(x,u_t (x),\nabla u_t (x))+ \mbox{div} g_t \left(x,u_t\left( x\right)
,\nabla u_t\left( x\right) \right) \, \big]\,  dt\\
& \hspace*{3cm} +  h_t(x,u_t(x),\nabla u_t(x))\cdot \overleftarrow{dB}_t  = - \nu (dt,x), \quad a.s.,  \\
& (iii')\;\;  \nu \big(u > v \big) =0, \quad a.s.,\\
 & (iv')  \; \;  u_T = \Psi, \quad d\mathbb{P}\otimes dx-\mbox{a.e.}.
\end{split}
\end{equation}
The main result in Matoussi and Stoica  \cite{MatoussiStoica} (see Theorem 4 and Corollary 2)  is the following which gives a probabilistic interpretation of the solution:
\begin{theorem}
Assume {\bf (H)}, {\bf (HD2)} and that the lower obstacle $v$ satisfies:
\begin{enumerate}
\item   $ v (t,\omega, x)$ is a predictable random function  with respect to the backward filtration  $\left(\mathcal{F}^B_{t,T}\right)_{t\in[0,T]}$, 
\item  $ t \mapsto v(t,\omega, W_t) $  is $\mathbb{P} \otimes \mathbb{P}^m$\textbf{-}a.e. continuous on $[0,T]$,
\item $v (T, \cdot) \leq \Psi (\cdot)$,
\item $ \mathbb{E}\mathbb{E}^m\left[ \sup_{t\in [0,T]}({v}^-(t,\cdot , W_t ))^{2} \right]<+\infty$.
\end{enumerate} 
Then OSPDE \eqref{OSPDE1} has a unique solution $u$ in $\mathcal{H}_T$.\\
Moreover, the triple of processes $\left(Y_t, Z_t, K_t \right)_{t \in [0,T]}$, the unique solution of the  following reflected backward doubly stochastic differential equation (in short RBDSDE) :
\begin{equation}
\begin{split}
\label{RBDSDE}
Y_{t} = \Psi\left( W_T\right) &+\int_{t}^{T}f_s
\left(W_{s},Y_{s},Z_{s}
\right) ds + K_T - K_t -\frac{1}{2}\int_{t}^{T}g_s\left(W_{s},Y_{s},
Z_{s}\right) *dW_s\\
& + \int_{t}^{T}h_s
\left(W_{s},Y_{s},Z_{s}
\right)\cdot  \overleftarrow{dB}_s -\sum\limits_{i}\int_{t}^{T}Z_{i,s}dW_{s}^{i} \,
\end{split}
\end{equation}
with $Y_t\geq L_t,\; \forall \, t \in [0,T] $,  $\left(K_t\right)_{t \in[0,T]}$ being an increasing continuous process, $K_0=0$ and    \begin{equation}
\label{reflection:minimum}
\displaystyle \int_{0}^{T}(Y_t-L_t)dK_t =  0
\end{equation}
is given by: $Y_t = u(t,W_t)$, $ Z_t =\nabla u(t,W_t )$, $L_t=v(t,W_t)$ and for any $\mathcal{\varphi }\in \mathcal{D}_T$, $$\mathcal{\nu }\left( \mathcal{\varphi }\right) =\mathbb{E}^{m}\int_{0}^{T}\mathcal{%
\varphi }\left( t,W_{t}\right) dK_{t}\,.$$\end{theorem}

\subsection{Approximation by the penalization method in the linear case}
\label{section:penalization}
We begin with the linear case, i.e. assume that $f$, $g$ and $h$ do not depend on $(u,\nabla u)$ . In other words we consider the following DOSPDE:
\begin{equation}
\label{OSPDElinear} \left\{\begin{split}& d u (t,x) +
 \;  \frac{1}{2} \Delta u (t,x)dt  + f_t (x)dt+ \mbox{div} g_t (x)dt +h_t (x)\cdot\overleftarrow{dB}_t +\nu^+ (dt,x)-\nu^- (dt,x)=0,\\
 &\underline{v}(t,x)\leq u(t,x)\leq \overline{v}(t,x),\\
 &\int_0^T\int_{\bbR^d} \left( \tilde{u}(t,x)-\underline{v}(t,x)\right)\nu^+ (dt,dx)=\int_0^T\int_{\bbR^d} \left( \overline{v}(t,x)-\tilde{u}(t,x)\right)\nu^- (dt,dx)=0,\\
 &u_T=\Psi, \end{split}\right. 
\end{equation}
where $f=f^0$, $g=g^0$, $h=h^0$ satisfy hypothesis {\bf (HD2)} and the obstacles $\underline{v}$ and $\overline{v}$ satisfy {\bf (HO)}.
\\For $n\in\mathbb{N}$, let  $(u^n, \nu^{+,n})$ be the solution of the following SPDE with lower obstacle:
\begin{equation}{\nonumber}
\label{SPDE:n} \left\{\begin{split}& du^n_t(x)+
 \frac{1}{2} \Delta u^n_t (x)dt +  f_t(x) dt  - n(u^n_t (x) -\overline{v}_t (x))^{+}dt 
 + \mbox{div} g_t(x)dt +  h_t(x)\cdot \overleftarrow{dB}_t+\nu^{+,n}(dt,x)= 0,\\
 & u^n(t,x)\geq \underline{v}(t,x),\\
 &\int_0^T\int_{\bbR^d}\left( \tilde{u}^n(t,x)-\underline{v}(t,x)\right)\nu^{+,n} (dt,dx)=0,\ \ \ \ \ \ \ \hspace{6.5cm} \eqref{SPDE:n}\\
 &u^n_T=\Psi. \end{split}\right. \\
\end{equation}
\addtocounter{equation}{1}

We denote  by $Y_t^n = u^n (t, W_t ) $, $ Z_t^n = \nabla u^n (t, W_t) $, $L_t=\underline{v}(t,W_t)$, $U_t = \overline{v} (t, W_t)$ and $\xi=\Psi(W_T)$. 
From  Theorem 4  in  Matoussi and Stoica \cite{MatoussiStoica} and for each $ n \in \mathbb N$,  there exists a unique quasi-continuous solution $u^n$ of the obstacle problem \eqref{SPDE:n}. Thus, $Y^n$ is $\mathbb{P}\otimes \mathbb{P}^m\textbf{-}a.e.$ continuous and   by  Corollary 2 in  \cite{MatoussiStoica},
the triplet $\big(Y^n, Z^n, K^{+,n} \big)$  solves the RBSDE  associated to the data $ (\Psi, f^n, g, h, L)$ with $f^n=f-n\,(Y^n-U)^+$, 
\begin{equation}
\label{BSDE:n}\left\{\begin{split}
&Y_{t}^n = \xi +\int_{t}^{T}f_s(W_s)ds - n \int_t^T (Y_s^n  - U_s)^{+}ds -\frac{1}{2}\int_{t}^{T}g_s*dW_s
+ \int_{t}^{T}h_s(W_s)\cdot\overleftarrow{dB}_s\\& \qquad\ -\sum\limits_{i}\int_{t}^{T}Z^n_{i,s}dW_{s}^{i} + K_T^{+,n}-K_t^{+,n},\\
&Y^n_t \geq L_t\,,\\
&\int_0^T (Y^n_t -L_t )dK^{+,n}_t =0\,.
\end{split}\right.\end{equation} 
From now on, we denote $K_{t}^{-,n}:=n\displaystyle \int_{0}^{t}(Y_{s}^{n}-U_{s})^{+}ds$. 

\begin{remark} $\quad$ \\
\begin{enumerate}
\item 
In \eqref{BSDE:n}, $(B_t)_{0\leq t\leq T}$ and $(W_t)_{0\leq t\leq T}$ are two mutually independent Brownian motions, with values respectively
in $\bbR^{d^1}$ and in $\bbR^d$. The backward filtration $\cF_{t,T}^B$ has been defined in Subsection \ref{DoublySto}
 and let $\cF_T^B:=\cF^B_{0,T}$. We also consider the following family of $\sigma-$fields $\cF_t^W:=\sigma(W_s,0\leq s\leq t)$. For any $t\in[0,T]$, we 
 define 
$$\cF_t:=\cF_t^W\vee\cF_{t,T}^B\ \ and\ \ \cG_t:=\cF_t^W\vee\cF_{T}^B\,.$$
Note that the collection $(\cF_t)_{t\in[0,T]}$ is neither increasing nor decreasing and it does not constitute a filtration. However, $(\cG_t)_{t\in[0,T]}$ is a filtration. 
\item  {\red{ Thanks to the comparison theorem (Lemma \ref{comparisonthmlinearRBDSDE} in the Appendix),  $(Y^n)_n$ is a non-increasing sequence since $f^n=f-n\,(Y^n-U)^+$is a non-increasing sequence}}. \\
\end{enumerate}
\end{remark}


We denote by $ \tilde Y_t = z(t, W_t ) $, $  \tilde Z_t = \nabla z (t, W_t)$ {\red{where $z$ satisfies the equation \eqref{obstacle}}}, then {\red from Theorem \ref{FK}}, 
$( \tilde Y,  \tilde Z)$ solves the following BDSDE: 
\begin{equation}\label{BDSDE:z}
 \tilde Y_{t} = \tilde \xi +\int_{t}^{T}  \tilde f_s(W_s)ds-\frac{1}{2}\int_{t}^{T} \tilde g_s*dW_s+ \int_{t}^{T}  \tilde h_s(W_s)\cdot\overleftarrow{dB}_s -\int_{t}^{T}  \tilde Z_{s}\,dW_s,
\end{equation}
where $ \tilde \xi=  \tilde \Psi (W_T)$. Moreover,   we have the following relation : $$L_t-  \tilde Y_t\leq-\kappa<0\leq U_t-\tilde Y_t.$$
\begin{lemma}
\label{penalization:estimate1}
There exists a constant $C$ independent of $n$ such that 
\begin{equation}\label{estimate1:n}
\mathbb{E}\mathbb{E}^m\sup_{t\in[0,T]}|Y_t^n|^2+\mathbb{E}\mathbb{E}^m\int_0^T|Z_t^n|^2dt + \mathbb{E}\mathbb{E}^m |K^{+,n}_T -K^{-,n}_T |^2\leq C.
\end{equation}
\end{lemma}
\begin{proof}
Applying It\^o's formula to $(Y^n- \tilde Y)^2$ (see Lemma \ref{itolowerobstacle}), for any $t\in[0,T]$, we have almost surely
\begin{equation}
\begin{split}
\label{itoynmoinsy2}
&|Y_t^{n}-\tilde  Y_t|^2 + \int_t^T |Z_s^{n}- \tilde Z_s|^2 ds =|\xi- \tilde  \xi|^2+ 2 \int_t^T (Y_s^{n}-\tilde  Y_s) (f_s(W_s)- \tilde  f_s(W_s))\,ds\\& - \int_t^{T} (Y_s^{n}-\tilde  Y_s) (g_s-\tilde  g_s)*dW_s+ 2 \int_{t}^{T} (Y_s^{n}-\tilde  Y_s) (h_s(W_s)- \tilde  h_s(W_s))\cdot\overleftarrow{dB}_s\\& -  2 \int_{t}^{T} (Y_s^{n}-\tilde  Y_s) (Z_{s}^{n}-\tilde  Z_s)\,dW_{s} - 2\int_t^T \langle Z_s^{n}- \tilde  Z_s, g_s(W_s)-\tilde  g_s(W_s) \rangle\, ds\\& + \int_t^T |h_s(W_s)-\tilde  h_s(W_s)|^2  ds+  2\int_t^T(Y_s^{n}-\tilde  Y_s)dK_s^{+,n} - 2n\int_t^T (Y_s^{n}-\tilde  Y_s) (Y_s^{n}-U_s)^+ds.
\end{split}
\end{equation}
From the Skorokhod condition \eqref{BSDE:n}, we get  $$\int_t^T(Y_s^{n}-\tilde  Y_s)dK_s^{+,n}=\int_t^T(L_s- \tilde  Y_s)dK_s^{+,n}\leq0$$ and 
\begin{equation*}\begin{split}
n\int_t^T(Y_s^{n}-\tilde  Y_s) (Y_s^{n}-U_s)^+ds&=n\int_t^T(Y_s^{n}-U_s+U_s-\tilde  Y_s) (Y_s^{n}-U_s)^+ds\\
&=\int_t^Tn((Y_s^n-U_s)^+)^2ds+\int_t^Tn(U_s-\tilde  Y_s)(Y_s^n-U_s)^+ds\geq0. 
\end{split}\end{equation*} Therefore, using Cauchy-Schwarz's inequality, trivial inequalities such as $2ab\leq 2a^2 +\displaystyle\frac{b^2}{2}$ and then  taking expectation under $\mathbb{P}\otimes\mathbb{P}^m$, we obtain
\begin{equation}
\begin{split}
\label{estimatynmoinsy2}
\mathbb{E}\mathbb{E}^m|Y_t^n-\tilde  Y_t|^2&+\mathbb{E}\mathbb{E}^m\int_t^T|Z_s^n-\tilde  Z_s|ds\leq \E|\xi- \tilde \xi|^2+\mathbb{E}\mathbb{E}^m\int_t^T|Y_s^n-\tilde  Y_s|^2ds+\frac12\mathbb{E}\mathbb{E}^m\int_t^T|Z_s^n-\tilde  Z_s|^2ds\\&+
\mathbb{E}\mathbb{E}^m\int_t^T\left[|(f_s-\tilde  f_s)(W_s)|^2+2|(g_s-\tilde  g_s)(W_s)|^2+|(h_s-\tilde  h_s)(W_s)|^2\right]ds.
\end{split}
\end{equation}
Hence,
$$\mathbb{E} \mathbb{E}^m|Y_t^{n}-\tilde  Y_t|^2\leq {C}+\mathbb{E} \mathbb{E}^m\int_t^T|Y_s^{n}-\tilde  Y_s|^2ds\, ,$$
where $C$ is a constant independent of $n$ which may vary from line to line.\\
From Gronwall's inequality it then follows that 
$$\sup_{0\leq t\leq T}\mathbb{E} \mathbb{E}^m|Y_t^{n}-\tilde  Y_t|^2\leq C$$
and again from \eqref{estimatynmoinsy2}, we have $$\mathbb{E} \mathbb{E}^m\int_0^T|Z_s^{n}-\tilde  Z_s|^2ds\leq C.$$
Coming back to \eqref{itoynmoinsy2} and using Bukholder-Davis-Gundy's inequality and the above estimates, we get $$\mathbb{E} \mathbb{E}^m\sup_{t\in[0,T]}|Y_t^{n}-\tilde  Y_t|^2\leq C.$$
Then, combining with the estimate for $(\tilde Y,\tilde Z)$ (see for example Theorem 2.1 in \cite{DenisStoica}), we obtain 
$$\mathbb{E}\mathbb{E}^m\sup_{t\in[0,T]}|Y_t^n|^2+\mathbb{E}\mathbb{E}^m\int_0^T|Z_t^n|^2dt\leq C.$$
Finally, we conclude since $$ K^{+,n}_T -K^{-,n}_T =Y_{0}^{n} - \xi -\int_{0}^{T}f_s(W_s)ds +\frac{1}{2}\int_{0}^{T}g_s*dW_s - \int_{0}^{T}h_s(W_s)\cdot\overleftarrow{dB}_s +\int_{0}^{T}Z^{n}_{s}dW_{s}\,.$$  
\end{proof}

Now we introduce a function $\psi\in C^2$ which satisfies $\psi(x)=x$ when $x\in(-\infty, -\kappa]$, $\psi(x)=0$ when $x\in[-\frac{\kappa}{2},+\infty)$.
\begin{lemma} The sequence $(K^{+,n}_T )_n$ is bounded in $L^1 (\mathbb{P}\otimes\mathbb{P}^m )$.
\end{lemma}
\begin{proof}
Applying It\^o's formula to $\psi(Y^n- \tilde Y)$, we have almost surely, $\forall t\in[0,T]$, 
\begin{equation}\label{itopsiy}
\begin{split}
\psi(Y_t^n-\tilde  Y_t)=&\,\psi(\xi- \tilde  \xi)+\int_t^T\psi'(Y_s^n- \tilde Y_s)dK_s^{+,n}-\int_t^T\psi'(Y_s^n-\tilde  Y_s)n(Y_s^{n}-U_s)^+ds\\&+\int_t^T\psi'(Y_s^n-\tilde  Y_s)(f_s(W_s)-\tilde  f_s(W_s))ds-\frac{1}{2}\int_t^T\psi'(Y_s^n- \tilde  Y_s)(g_s-\tilde g_s)*dW_s\\&+\int_t^T\psi'(Y^n_s-\tilde  Y_s)(h_s(W_s)-\tilde  h_s(W_s))\cdot\overleftarrow{dB}_s-\int_t^T\psi'(Y_s^n-\tilde  Y_s)(Z_s^n-\tilde  Z_s)dW_s\\&-\frac{1}{2}\int_t^T\psi''(Y_s^n-\tilde  Y_s)|Z_s^n-\tilde  Z_s|^2ds+\frac{1}{2}\int_t^T\psi''(Y_s^n- \tilde Y_s)|h_s(W_s)- \tilde h_s(W_s)|^2ds\\&-\int_t^T\psi''(Y_s^n-\tilde  Y_s)\langle g_s(W_s)-\tilde  g_s(W_s),Z_s^n- \tilde  Z_s\rangle ds.
\end{split}
\end{equation}
We note that $$\int_t^T\psi'(Y_s^n-\tilde  Y_s)dK_s^{+,n}=\int_t^T\psi'(L_s-\tilde  Y_s)dK_s^{+,n}=K_T^{+,n}-K_t^{+,n}$$
and $$\int_t^T\psi'(Y_s^n-\tilde Y_s)n(Y_s^{n}-U_s)^+ds=0,$$
then, combining with Lemma \ref{penalization:estimate1} and using the fact that $\psi'$ and $\psi''$ are bounded, we deduce that, there exists a constant $c > 0$  such that 
\begin{equation}\label{boundnessK+n}
\mathbb{E}\mathbb{E}^m[K_T^{+,n}]\leq c.
\end{equation} 
 \end{proof}
Thus  the sequence of processes $K^{+,n}_t$ is  uniformly bounded in $L^1 (\Omega\times \Omega'\times [0,T] )$. Moreover, 
by comparison theorem (see Lemma \ref{comparisonthmlinearRBDSDE}), we have $dK^{+,n+1}\geq dK^{+,n}$.  \\
Therefore the sequence $(K^{+,n}_t)_n $   converges to a  process denoted  by $K^+$  and by Fatou's lemma, we get 
\begin{equation}\label{boundnessK+n}
\mathbb{E}\mathbb{E}^m[K_T^{+}]\leq c.
\end{equation}  
Moreover,  by Lemma 3.2 in \cite{PengXu},  we have the following result:
\begin{lemma}{\label{contK+}}
$(K_t^{+})_{t \geq 0} $ is an increasing and continuous process.
\end{lemma}
Moreover, by the comparison theorem (see Lemma \ref{comparisonthmlinearRBDSDE}), 
we know  that $(Y^n )_{n}$ is non-increasing and bounded in $L^2$, as a consequence it converges to a process  that we denote by $Y$ and we have
$$\lim_{n\rightarrow +\infty}\E\E^m \left[\int_0^T |Y^n_t -Y_t|^2 dt\right]=0.$$We are now going to prove that the process $(Y_t )_{t\in[0,T]}$ admits a continuous version which solves \eqref{DRBDSDE}.

\subsection{The fundamental lemma}\label{FundLemma}
We recall that for all $n\in\mathbb{N}$,
\begin{equation*}
\begin{split}
Y_{t}^n =&\, \xi +\int_{t}^{T}f_s(W_s)ds - n\int_t^T (Y_s^n  - U_s)^{+}\, ds -\frac{1}{2}\int_{t}^{T}g_s*dW_s
+ \int_{t}^{T}h_s(W_s)\cdot\overleftarrow{dB}_s\\& -\sum\limits_{i}\int_{t}^{T}Z^n_{i,s}dW_{s}^{i} + K_T^{+,n}-K_t^{+,n}.\end{split}
\end{equation*} 

First of all, by extracting a subsequence if necessary, we can assume that $(Z^n)_{n}$ converges weakly to a process $Z$ in $L^2 (\Omega \times \Omega'\times[0,T])$. \\
Let $\zeta$ be a positive function in $L^2 (\R^d)$, 
we introduce, for each $N>0$, the $\mathcal{G}_t-$stopping time
$$\tau_{N }:=\inf \left\{ t\geq 0, \, K_{t}^+=N \zeta (W_0)\right\}\wedge T.$$
Since  $K^{+}_{T}$ is integrable and $\zeta>0$, $\mathbb{P}\otimes\mathbb{P}^m$-a.e.,  $\tau_N =T$ for $N$ large enough. Moreover, as $K^+_{\tau_N}\leq N\zeta (W_0)\wedge K^+_T$, $K^+_{\tau_N}$ belongs to $L^2 (\mathbb{P}\otimes\mathbb{P}^m )\bigcap L^1 (\mathbb{P}\otimes \mathbb{P}^m )$.
Clearly for each $N$,  $(Z^n_{\cdot\wedge \tau_{N}})_{n}$ converges weakly to  $Z_{\cdot\wedge\tau_{N}}$ in $L^2 (\Omega \times \Omega'\times[0,T])$.\\
We have
\begin{equation}\label{forwardBSDE:n}\begin{split}
Y_{t\wedge \tau_{N}}^{n} =&\, Y^n_{0}  -\int_{0}^{t\wedge \tau_{N}}f_s (W_s)ds+\frac12 \int_{0}^{t\wedge \tau_{N}}g_s * dW_s-\int_{0}^{t\wedge \tau_{N}}h_s(W_s)\cdot\overleftarrow{dB}_s\\&+\int_{0}^{t\wedge \tau_{N}}Z^{n}_{s}dW_{s}-K_{t\wedge \tau_{N}}^{+,n} + n \int_0^{t\wedge \tau_{N}}(Y_s^{n} - U_s)^{+}\, ds\,.\end{split}\end{equation}

Let us remark that for fixed $N$,  the sequence of processes $(K^{+,n}_{\cdot\wedge\tau_{N}})$ is non-decreasing in $n$, bounded in $L^2$ and for each $n$ and $0\leq s\leq t\leq T$, $K^{+,n+1}_{t\wedge\tau_{N}}-K^{+,n+1}_{s\wedge\tau_{N}}\geq K^{+,n}_{t\wedge\tau_{N}}-K^{+,n}_{s\wedge\tau_{N}}$ by the comparison theorem. Moreover, by Lemma \ref{penalization:estimate1} this implies that the sequence of processes $(K^{-,n}_{\cdot\wedge\tau_{N}})$ is also bounded in $L^2$.\\
Then, as a consequence of the stochastic monotonic convergence theorem due to Peng and Xu (see Theorem 3.1 in \cite{PengXu}) we conclude that there exists an increasing process $K^-$ such that, for each $N>0$ and $t\in [0,T]$, $Y_{t\wedge \tau_{N}}^{n}$ converges to 
\begin{equation}\label{Continu}\begin{split}Y_{t\wedge\tau_{N}}=&\,Y_{0} -\int_{0}^{t\wedge \tau_{N}}f_s (W_s)ds-\frac12 \int_{0}^{t\wedge \tau_{N}}g_s * dW_s-\int_{0}^{t\wedge \tau_{N}}h_s(W_s)\cdot\overleftarrow{dB}_s\\&+\int_{0}^{t\wedge \tau_{N}}Z_{s}dW_{s}-K_{t\wedge \tau_{N}}^{+} +K_{t\wedge \tau_{N}}^{-},\end{split}\end{equation}
where both $K_{\cdot\wedge \tau_{N}}^{+}$ and $K_{\cdot\wedge \tau_{N}}^{-}$ hence $Y_{\cdot\wedge\tau_{N}}$ are c\`adl\`ag. Making $N$ tend to infinity we conclude that $Y$ and $K^-$ are c\`adl\`ag.\\
\begin{lemma}{\label{FundamentalLemma}}(Fundamental Lemma)
We have
$$\lim_{n\rightarrow +\infty}\mathbb{E}\mathbb{E}^m \bigg[\Big(\sup_{t\in [0,T]}(Y^n_{t} -U_{t})^+\Big)^2\bigg]=0.$$
\end{lemma}
\begin{proof}  
We first note that for all $n$, $Y^n-  \tilde Y$ is the solution of  the RBSDE  associated to the data $ (\Psi (W_T) - \tilde \Psi (W_T), f^n - \tilde f, g - \tilde g, h - \tilde h, L - \tilde Y)$. By It\^o's formula (see Remark \ref{remarkitoforpositivepart} in the appendix), we have
\begin{equation}
\label{ItoLF}
\begin{split}
&|(Y_0^{n}- \tilde  Y_0)^+|^2 = |(\xi -  \tilde  \xi )^+|^2-\int_0^T{\bf{1}}_{\{ Y^n_s- \tilde Y_s >0\}} |Z_s^{n}-  \tilde Z_s|^2 ds - 2 \int_{0}^{T} (Y_s ^{n} - \tilde  Y_s)^+ (Z_{s}^{n}-  \tilde Z_s)dW_{s} \\&+ 2\int_0^T (Y_s^{n}-  \tilde Y_s)^+\,dK_s^{+,n}+2\int_0^T(Y_s^n-  \tilde Y_s)^+(f_s(W_s)-  \tilde f_s(W_s))ds
-\int_0^T (Y_s ^{n} -  \tilde Y_s)^+ (g_s -  \tilde g_s)*dW_s\\&+2\int_0^T (Y_s ^{n} -  \tilde Y_s)^+ (h_s(W_s) - \tilde  h_s(W_s) )\cdot\overleftarrow{dB}_s+\int_0^T{\bf{1}}_{\{ Y^n_s-  \tilde Y_s >0\}} |h_s(W_s)-  \tilde  h_s(W_s)|^2 ds\\&-2\int_0^T{\bf{1}}_{\{ Y^n_s-  \tilde Y_s >0\}} \langle g_s(W_s) -  \tilde g_s(W_s) ,Z^n_s -  \tilde  Z_s \rangle ds - 2n\int_0^T (Y_s^{n}-  \tilde Y_s)^+ (Y_s^{n}-U_s)^+ds,\ a.s.. 
\end{split} \end{equation}
 Since $ \tilde  Y \geq L$, we have $\int_0^T (Y_s^{n}- \tilde  Y_s)^+\,dK_s^{+,n}\leq 0$.
 Now taking the expectation  in the previous inequality, we easily get that $n  \int_0^T(Y_s^{n}-  \tilde Y_s)^+ (Y_s^{n}-U_s)^+ ds$ is bounded in $L^1$, which yields: 
 $$\lim_{n\rightarrow +\infty}\mathbb{E}\mathbb{E}^m \left[\int_0^T(Y_s^{n}- \tilde  Y_s)^+ (Y_s^{n}-U_s)^+ ds\right]=0.$$
 Hence, since $ \tilde  Y \leq U$,
 $$\lim_{n\rightarrow +\infty}\mathbb{E}\mathbb{E}^m\left [\int_0^T\big ((Y_s^{n}-U_s)^+\big)^2 ds\right]=0.$$
 So we have $\mathbb{E}\mathbb{E}^m [\int_0^T ((Y_s-U_s)^+)^2 ds]=0.$\\
 Since $Y$  is c\`adl\`ag process and $U$ is continuous process, we deduce that $\mathbb{P}\otimes\mathbb{P}^m$-a.e., for all $t\in [0,T]$, $Y_t \leq U_{t}$. Hence $\lim_{n\rightarrow +\infty} (Y^n_{t}-U_{t})^+=0$. 
 By Dini's lemma (\cite{DellacherieMeyerII}, p. 202), we get $\lim_{n\rightarrow +\infty} \sup_{t\in [0,T]}(Y^n_{t}-U_{t})^+=0$, a.s. and we conclude by the dominated convergence theorem.
 \end{proof}

\begin{lemma}\label{boundnessofK} For each $N>0$, there exists a constant $C_N$ such that for all $n\in\N$, $$\mathbb{E}\mathbb{E}^m [(K^{+,n}_{\tau_N})^2]+\mathbb{E}\mathbb{E}^m [(K^{-,n}_{\tau_N})^2]\leq C_N.$$
\end{lemma}
\begin{proof} This is an obvious consequence of Lemma \ref{penalization:estimate1} and the fact that $K^{+,n}$ is dominated by $K^+$.
\end{proof}
Since we have the following estimate (see Lemma \ref{penalization:estimate1})
\begin{equation}\label{mainestimate2}
\mathbb{E} \mathbb{E}^m \left(\sup_{0\leq t\leq T} |Y_t^n|^2\right)+  \mathbb{E }\mathbb{E}^m  \int_0^T |Z_s^n|^2ds  \leq C,
\end{equation}
by Fatou's lemma , one gets  $$ \mathbb{E}\mathbb{E}^m \left( \sup_{0 \leq t \leq T}  |Y_t|^2 \right) \leq  C.$$

\subsection{ Convergence of $(Y^n, Z^n, K^{+,n},K^{-,n})$}
 
\begin{lemma}
\label{convergence:YZK} The limiting processes $Y$, $K^+$ and $K^- $ are continuous and
we have:
\begin{equation}
\lim_{n\rightarrow +\infty}\mathbb{E}\mathbb{E}^m  \left[ \sup_{0 \leq t \leq T}  |Y_t^n - Y_t|^2 \right] = 0\,, \end{equation}
and for any $N>0$,
 \begin{equation}
\label{convergence:YZKn}
\lim_{n\rightarrow +\infty}\mathbb{E}\mathbb{E}^m  \left[ \int_0^{\tau_{N}}  |Z_t^n - Z_t|^2  dt \right] = 0\,. \end{equation}
Moreover, we have   $L_t\leq Y_t\leq U_t,\; \forall \, t \in [0,T] $ and   
$$\int_0^T(Y_s-L_s)dK^+_s=\int_{0}^{T}(U_s-Y_s)dK^-_s =  0, \quad \mathbb{P}\otimes \mathbb{P}^m\textbf{-}a.e..$$

\end{lemma}

\begin{proof} The continuity of process $K^+$ has been proved in Lemma \ref{contK+}.\\
From \eqref{forwardBSDE:n}, we have, for $n\geq p$ and any $t\in[0,T]$,  
\begin{equation}\label{Yn-Yp}
Y^n_{t\wedge\tau_N}-Y^p_{t\wedge\tau_N}=(Y^n_0-Y^p_0)-(K^{+,n}_{t\wedge\tau_N}-K^{+,p}_{t\wedge\tau_N})+(K^{-,n}_{t\wedge\tau_N}-K^{-,p}_{t\wedge\tau_N})+\int_0^{t\wedge\tau_N}(Z^n_s-Z^p_s)dW_s.
\end{equation}
Then, It\^o's formula gives almost surely, 
\begin{equation}
\label{Ito:BSDEnp}
\begin{split}
 |Y_{t\wedge\tau_N}^n -Y_{t\wedge\tau_N}^p|^2 =&\, |Y^n_0-Y^p_0|^2- 2 \int_0^{t\wedge\tau_N} \left(Y_s^n - Y_s^p\right) d\left(K_s^{+,n} - K_s^{+,p}\right)
\\&+ 2\int_0^{t\wedge\tau_N} \left(Y_s^n - Y_s^p\right) d\left(K_s^{-,n} - K_s^{-,p}\right) +\int_0^{t\wedge\tau_N}|Z_s^n-Z_s^p|^2ds
\\&+2\sum\limits_{i}\int_0^{t\wedge\tau_N} \left(Y_s^n - Y_s^p\right) \left(Z_{i,s}^n - Z_{i,s}^p \right)dW_{s}^{i}. 
\end{split}
\end{equation}

Then taking expectation and noting that, as $n\geq p$, due to the comparison theorem (Lemma \ref{comparisonthmlinearRBDSDE}  in the Appendix), we know that $Y^n\leq Y^p$, hence {\red{thanks to the definition of $ K^{-,n}$}}, we get
{\red
\begin{equation}\label{controlthirdterm}
\begin{split}
-2\int_0^{t\wedge\tau_N} \left(Y_s^n - Y_s^p\right) d\left(K_s^{-,n} - K_s^{-,p}\right)
\leq&\, -2\int_0^{t\wedge\tau_N} \left(Y_s^n - Y_s^p\right) dK_s^{-,n}\\=&-2\int_0^{t\wedge\tau_N} \left(Y_s^n - U_s\right) dK_s^{-,n}
+2\int_0^{t\wedge\tau_N} \left(Y_s^p - U_s\right) dK_s^{-,n}\\
\leq&\,2\int_0^{t\wedge\tau_N} \left(Y_s^p - U_s\right) dK_s^{-,n}\\
=&\,2\int_0^{t\wedge\tau_N} \left[(Y_s^p - U_s)^+-(Y^p_s-U_s)^-\right]n(Y_s^n-U_s)^+ds\\
{ \leq}&\,2\int_0^{t\wedge\tau_N} \left(Y_s^p - U_s\right)^+ dK_s^{-,n}\\
\leq&\,2 \sup_{s\in [0,T]}\left(Y_s^p -  U_s\right)^+ { K^{-,n}_{\tau_N}}.
\end{split}\end{equation}

By Lemma \ref{FundamentalLemma}, Cauchy-Schwarz's inequality and the fact that $(K^{-,n}_{\tau_N})_n$ is bounded in $L^2$, we get 
\begin{eqnarray*}
\lim_{n{,p}\rightarrow +\infty}\bbE\bbE^m\left[ \sup_{s\in [0,T]}\left({ Y_s^p} -  U_s\right)^+ {K^{-,n}_{\tau_N}}\right]  =0,\end{eqnarray*}
hence
\begin{eqnarray}\label{controltermsup}
\limsup_{n,p\rightarrow +\infty}\E\E^m \left[ \sup_{t\in [0,T]}\left( -\int_0^{t\wedge\tau_N} \left(Y_s^n - Y_s^p\right) d\left(K_s^{-,n} - K_s^{-,p}\right) \right)\right]  \leq 0.\end{eqnarray}
Moreover, for fixed $N$, the sequence $(Y^n_{t\wedge \tau_N})_{n\in\mathbb{N}}$ is decreasing and bounded in $L^2$ hence converges in $L^2$, so 
$$\lim_{n,p\rightarrow +\infty}\E\E^m\left [ \left|Y^n_{t\wedge \tau_N}-Y^p_{t\wedge \tau_N}\right|^2\right]=0.$$

Finally, remarking the following relation: 
\begin{equation*}\begin{split}
&2\int_0^{t\wedge\tau_N}(Y_s^n-Y_s^p)d(K_s^{+,n}-K_s^{+,p})\\=&\,2\int_0^{t\wedge\tau_N}(Y_s^n-L_s+L_s-Y_s^p)d(K_s^{+,n}-K_s^{+,p})\\=&\,2\int_0^{t\wedge\tau_N}(Y_s^n-L_s)d(K_s^{+,n}-K_s^{+,p})
+2\int_0^{t\wedge\tau_N}(Y_s^p-L_s)d(K_s^{+,p}-K_s^{+,n})\\=&\,2\int_0^{t\wedge\tau_N}(L_s-Y_s^n)dK_s^{+,p}+2\int_0^{t\wedge\tau_N}(L_s-Y_s^p)dK_s^{+,n}\leq 0,
\end{split}\end{equation*}
and coming back to \eqref{Ito:BSDEnp}, we get for any $t\in [0,T]$,
\begin{equation*}\label{convergenceofZnKn+}
0\leq \limsup_{n,p\rightarrow +\infty}\bbE\bbE^m\int_0^{t\wedge\tau_N}|Z_s^n-Z_s^p|^2ds\leq\limsup_{n,p\rightarrow +\infty}-2\bbE\bbE^m\int_0^{t\wedge\tau_N}(Y_s^n-Y_s^p)d(K_s^{-,n}-K_s^{-,p})\leq0. 
\end{equation*}

Finally,    taking supremum over $[0,T]$ in \eqref{Ito:BSDEnp}, thanks to the Burkholder-Davis-Gundy inequality and \eqref{controltermsup}, we have }
\begin{equation}\label{cauchy:YZ}
\mathbb{E}\mathbb{E}^m\left[\sup_{t\in[0,\tau_N]}|Y_t^n-Y_t^p|^2+\int_0^{\tau_N} |Z_t^n - Z_t^p|^2 dt\right] \longrightarrow 0, \quad  \mbox{as} \; \,     n,p \to \infty.\end{equation}

Hence, there exists a pair $(Y, Z)$ of progressively measurable processes with values in $ \mathbb{R}\times\mathbb{R}^d$ such that 
\begin{equation*}
\mathbb{E}\mathbb{E}^m  \left[ \sup_{0 \leq t \leq \tau_N}  |Y_t^n - Y_t|^2+ \int_0^{\tau_N} |Z_t^n - Z_t|^2  dt  \right] \longrightarrow 0.\end{equation*}
Therefore, the process $Y$ admits a continuous version that we still denote by $Y$ and by equality \eqref{Continu} we deduce that the process 
$K^-$ is also continuous. Then as a consequence of Dini's lemma, $\sup_{t\in [0,T]}|Y^n_t -Y_t|$ converges to $0$ a.e., hence by the dominated convergence theorem, we have
$$\lim_{n\rightarrow+\infty}\E\E^m\left [\sup_{t\in [0,T]}|Y^n_t -Y_t|^2\right ]=0.$$\\
Similarly, Dini's lemma and the dominated convergence theorem also yield
$$\E\E^m\left[\sup_{t\in [0,T]}\left|K^{+,n}_t -K_t^+\right|\right]=0.$$
From identity \eqref{Yn-Yp}, making $p$ tend to $+\infty$, we have
\begin{equation*}\begin{split}
\sup_{t\in [0,T]}|K^{-,n}_{t\wedge\tau_N}-K^{-}_{t\wedge\tau_N}|\leq&\,|Y^n_0-Y_0|+\sup_{t\in [0,T]}|Y^n_{t}-Y_{t}|+\sup_{t\in [0,T]}|K^{+,n}_{t}-K^{+}_{t}|\\&+\sup_{t\in [0,T]}\left|\int_0^{t\wedge\tau_N}(Z^n_s-Z_s)dW_s\right|.
\end{split}\end{equation*}
We have proved that each term on the right hand side tends to $0$ in $L^2$ or $L^1$, hence by standard arguments based on a diagonal extraction procedure we can extract a subsequence  $(K^{-,\delta (n)})_n$ such that for all $N$,
$$\lim_{n\rightarrow +\infty} \sup_{t\in [0,T]}|K^{-,\delta (n)}_{t\wedge\tau_N}-K^{-}_{t\wedge\tau_N}|=0,\ \ \makebox{a.e.}.$$
Hence, since $\tau_N =T$ a.e. for $N$ large enough, we obtain
$$\lim_{n\rightarrow +\infty} \sup_{t\in [0,T]}|K^{-,\delta (n)}_{t}-K^{-}_{t}|=0,\ \ \makebox{a.e.}.$$

Moreover, from  Lemma \ref{FundamentalLemma} we know that $\mathbb{P}\otimes \mathbb{P}^m \mathrm{-} a.e.$,
$Y_t\leq U_t,\ \forall \, t \in [0,T]$, which yields that $\int_{0}^{T}(Y_s-U_s)dK_s^- \leq 0$. But, we also have
\begin{equation*}\begin{split}
\left |\int_0^T(Y_s-U_s)dK_s^- -\int_0^T(Y_s^n-U_s)dK_s^{-,n}\right|\leq&\,\left|\int_0^T (Y_s-U_s)dK_s^- -\int_0^T (Y_s-U_s)dK_s^{-,n}\right| \\
 &+\int_0^T |Y_s -Y^n_s |dK_s^{-,n}.
\end{split} \end{equation*}
 Now since $K^{-,\delta (n)}$ tends to $K^-$, we deduce that almost surely, the sequence  $(dK^{-,\delta (n)})_n$ of measures on $[0,T]$ converges weakly to $dK^-$. Since $s\rightarrow Y_s -U_s$ is continuous, we have 
$$\lim_{n\rightarrow +\infty}\int_0^T(Y_s-U_s)dK_s^{-,\delta (n)}=\int_0^T(Y_s-U_s)dK_s^-.$$
Then, for all $n,N>0$,
$$\int_0^{ \tau_N} |Y_s -Y^{\delta (n)}_s |dK_s^{-,\delta (n)}\leq \sup_{t\in [0,T]} |Y_t -Y^{\delta (n)}_t |K_{\tau_N}^{-,{\delta (n)}}.$$
By Cauchy-Schwarz's inequality and Lemma \ref{boundnessofK}, $\lim_{n\rightarrow +\infty}\E\E^m [\sup_{t\in [0,T]} |Y_t -Y^n_t |K_{\tau_N}^{-,n}]=0$, so by extracting another subsequence if necessary we have
$$\lim_{n\rightarrow +\infty}\int_0^{ \tau_N} |Y_s -Y^{\delta (n)}_s |dK_s^{-,\delta (n)}=0,\ \ \makebox{ a.e.,}$$
this yields
\begin{equation*}
\int_0^{\tau_N}(Y_s-U_s)dK_s^-=\lim_{n\rightarrow +\infty}\int_0^{\tau_N}(Y_s^{\delta (n)}-U_s)dK_s^{-,\delta (n)}=\lim_{n\rightarrow +\infty}\int_0^{\tau_N}\delta (n)((Y_s^{\delta (n)}-U_s)^+)^2ds\geq 0. 
\end{equation*}
Hence, $\int_{0}^{\tau_N}(Y_s-U_s)dK_s^- = 0$. Using similar arguments, we prove
\begin{equation*}
\int_0^{\tau_N}(Y_s-L_s)dK_s^+=\lim_{n\rightarrow +\infty}\int_0^{\tau_N}(Y_s^n-L_s)dK_s^{+,n}=0. 
\end{equation*}

Finally, as $\tau_N=T$ almost surely for $N$ large enough, we conclude that 
$$\int_0^T(Y_s-L_s)dK^+_s=\int_{0}^{T}(U_s-Y_s)dK^-_s =  0, \quad \mathbb{P}\otimes \mathbb{P}^m\textbf{-}a.e..$$
\end{proof}

As a consequence of the  last proof, by passing to the limit in \eqref{forwardBSDE:n} a.e., we obtain the following generalization of the DRBSDE introduced in \cite{EKPPQ}:
\begin{corollary}
\label{RBDSDE:definition}
 The limiting quadruple of processes $\left(Y_t, Z_t, K_t^+,K_t^- \right)_{t \in [0,T]}$ is a solution of the  following DRBDSDE:
\begin{equation}
\begin{split}
Y_{t}& = \Psi\left( W_T\right) +\int_{t}^{T}f_r
\left(W_{r}\right) dr -\frac{1}{2}\int_{t}^{T}g_r *dW_r  + \int_{t}^{T}h_r
\left(W_{r}\right)\cdot\overleftarrow{dB}_r \\&\quad\quad -\sum\limits_{i}\int_{t}^{T}Z_{i,r}dW_{r}^{i}+ K_T^+ - K_t ^+-K_T^-+K_t^- \,
\end{split}
\end{equation}
with $L_t\leq Y_t\leq U_t,\; \forall \, t \in [0,T] $,  $\left(K_t^+\right)_{t \in[0,T]}$ and $\left(K_t^-\right)_{t \in[0,T]}$ are increasing continuous processes and    \begin{equation}
\label{reflection:minimum}
\displaystyle \int_{0}^{T}(Y_t-L_t)dK^+_t = \int_0^T(U_t-Y_t) dK_t^-= 0.
\end{equation}
\end{corollary}
\subsection{ End of the proof of Theorem \ref{maintheorem} in the linear case} 
At this stage, we have built the solution of the DRBDSDE associated to our DOSPDE. It remains to make the link with the solution of this DOSPDE.\\
We keep all the notations of the preceding section.
\begin{lemma}There exists $u\in \mathcal{H}_T$ which admits a quasi-continuous version that we still denote by $u$ such that
$$\forall t\in [0,T],\, Y_t =u(t,W_t )\ \makebox{ and }\ Z_t =\nabla u (t,W_t ),\ \mathbb{P}\otimes\mathbb{P}^m-\makebox{a.e.}.$$
\end{lemma}
\begin{proof}
First of all, as a consequence of \cite{MatoussiStoica}, for each $n\in\N$ there exists $u^n\in\mathcal{H}_T$, quasi-continuous, such that 
$$\forall t\in [0,T],\, Y^n_t =u^n(t,W_t )\makebox{ and }Z^n_t =\nabla u^n (t,W_t ),\ \mathbb{P}\otimes\mathbb{P}^m-\makebox{a.e.}.$$
Since the sequence $(Z^n)$ is bounded in $L^2$, {\blue by Mazur's lemma}, we can construct a sequence of convex combination
$$\tilde{Z}^n:=\sum_{i\in I_n}\alpha^n_i Z^i,$$
which converges to $Z$ in $L^2 ([0,T]\times\Omega\times\Omega')$.
We put 
$$\tilde{Y}^n :=\sum_{I_n}\alpha^n_i Y^i\ \ \makebox{and}\ \ \tilde{u}^n :=\sum_{I_n}\alpha^n_i u^i,$$
then we clearly have
\begin{equation*}
\sup_{t\in [0,T]}  \| \tilde{u}_t^n - \tilde{u}_t^p \|^2 + \int_0^T\|\nabla  \tilde{u}_t^n - \nabla \tilde{u}_t^p\|^2\,dt \leq \mathbb{E}^m \left[\sup_{t\in [0,T]} | \tilde{Y}_t^n - \tilde{Y}_t^p|^2 + \int_0^T| \tilde{Z}_t^n - \tilde{Z}_t^p|^2\,dt\right],
\end{equation*}
so the sequence $\big(\tilde{u}^n\big)_{n\in \mathbb{N}}$ is a Cauchy sequence in 
$\mathcal{H}_T$ and hence has a limit $u$ in this space.
The end of the proof is then obvious.
\end{proof}

Finally, we have the desired result:
\begin{lemma} The triple $(u,\nu^+ ,\nu^-)$, where 
$$\forall \varphi \in C_b ([0,T]\times \R^d),\ \int_0^T \int_{\mathbb{R}^d} \varphi (t,x) \, \nu^+\big(dt,dx\big) = \E^m \int_0^T \varphi_t(W_t) \,dK_t ^+$$
and
$$\forall \varphi \in C_b ([0,T]\times \R^d),\ 
\int_0^T \int_{\mathbb{R}^d} \varphi (t,x) \, \nu^-\big(dt,dx\big) = \E^m \int_0^T \varphi_t(W_t) \,dK_t ^- ,$$  
is a solution of the linear SPDE with two obstacles \eqref{OSPDElinear} .
\end{lemma}
\begin{proof} Let $(Y^1 ,Z^1)$ be the solution of the backward doubly SDE without obstacle
\begin{equation*}
\begin{split}
Y_{t}^1& = \Psi\left( W_T\right) +\int_{t}^{T}f_r
\left(W_{r}\right) dr -\frac{1}{2}\int_{t}^{T}g_r *dW_r  + \int_{t}^{T}h_r
\left(W_{r}\right)\cdot  \overleftarrow{dB}_r  -\sum\limits_{i}\int_{t}^{T}Z_{i,r}^1dW_{r}^{i}\,,
\end{split}
\end{equation*}
then we know that $Y^1_t=u^1(t,W_t)$, $Z^1_t =\nabla u^1(t,W_t)$, where $u^1\in\mathcal{H}_T$ is quasi-continuous and the solution of the SPDE:
$$d u^1 (t,x) +\frac{1}{2} \Delta u^1 (t,x)dt  + f_t (x)dt+ \mbox{div} g_t (x)dt +h_t (x)\cdot\overleftarrow{dB}_t =0$$
 with terminal condition $u^1_T =\Psi.$\\
 We put $Y^2=Y-Y^1$, $Z^2=Z-Z^1$, then $(Y^2 ,Z^2 ,K^+ ,K^-)$ is a solution of the following backward doubly stochastic SDE with lower obstacle 
 $L_t -Y^1_t$ and upper obstacle $U_t -Y^1_t$:
 $$Y^2_{t} =  -\sum\limits_{i}\int_{t}^{T}Z^2_{i,r}dW_{r}^{i}+ K_T^+ - K_t ^+-K_T^-+K_t^- .$$
 Now, we put $u^2 =u-u^1$, then $u^2 \in\mathcal{H}_T$ is quasi-continuous, and moreover,
 $$ Y^2_t =u^2 (t,W_t )\makebox{ and }Z^2_t=\nabla u^2 (t,W_t ).$$
 Let $\varphi\in \mathcal{D}_T$, then by It\^o's formula, we have for all $t\in [0,T]$:
 \begin{equation*}\begin{split}
 \varphi_t(W_t )Y^2_t =&\,-\int_t^T \left( Y^2_s\nabla \varphi_s(W_s) +\varphi_s(W_s)Z^2_s\right)dW_s -\int_t^T \partial_s \varphi_s(W_s)Y^2_sds-\frac12 \int_t^T \Delta \varphi _s(W_s )Y^2_sds\\& -\int_t^T \nabla \varphi_s(W_s)\cdot Z^2_s ds+\int_t^T \varphi_s(W_s)\, dK^{+}_s -\int_t^T \varphi_s(W_s )\, dK^-_s,
\end{split}\end{equation*}
then taking expectation w.r.t. $\E^m$ and remarking that for example
$$ \E^m \int_t^T \Delta \varphi_s (W_s )Y^2_s\,ds = \int_t^T \left( \Delta \varphi_s ,u^2_s \right)ds=-\int_t^T \left( \nabla \varphi_s ,\nabla u^2_s \right)ds,$$
we get
\begin{equation*}
\begin{split}
& \int_{t}^{T } \left[\left( u^2_{s},\partial_{s}\varphi_s \right) +
\frac{1}{2} \left( \nabla u^2_{s},\nabla \varphi_s \right)  \right] ds+\big( u^2_t, \varphi_t \big) =
 \int_{t}^{T}\int_{\mathbb{R}^d}\varphi_s(x)\, (\nu^+ -\nu^-) (ds,dx),\ \ a.s..\\
\end{split}
\end{equation*}
This proves that $(u^2 ,\nu^+ ,\nu^- )$ solves 
\begin{equation*}
 \left\{\begin{split}& d u^2 (t,x) +
 \;  \frac{1}{2} \Delta u^2 (t,x)dt  + \nu^+ (dt,x)-\nu^- (dt,x)=0,\\
 &\underline{v}(t,x)-u^1 (t,x)\leq u^2(t,x)\leq \overline{v}(t,x)-u^1 (t,x),\\
 &\int_0^T\int_{\bbR^d} \left( {u^2}(t,x)-(\underline{v}(t,x)-u^1 (t,x))\right)\nu^+ (dt,dx)=\int_0^T\int_{\bbR^d} \left( (\overline{v}(t,x)-u^1 (t,x))-{u^2}(t,x)\right)\nu^- (dt,dx)=0,\\
 &u^2_T=0.\end{split}\right. 
\end{equation*}
It is now easy to conclude since $u=u^1 +u^2$.
\end{proof}

The next proposition will ensure the uniqueness of solution.
\begin{proposition} Let $(u,\nu^+ ,\nu^-)$ be a solution of the linear SPDE with two obstacles \eqref{OSPDElinear}. We consider that $u$ is the quasi-continuous version and  denote by $K^+$ and $K^-$ the processes such that:
$$\forall \varphi \in \mathcal{C}_b ([0,T]\times \R^d),\ \ \int_0^T \int_{\mathbb{R}^d} \varphi (t,x) \, \nu^+\big(dt,dx\big) = \E^m \int_0^T \varphi_t(W_t) \,dK_t ^+$$
and
$$\forall \varphi \in \mathcal{C}_b ([0,T]\times \R^d),\ \ 
\int_0^T \int_{\mathbb{R}^d} \varphi (t,x) \, \nu^-\big(dt,dx\big) = \E^m \int_0^T \varphi_t(W_t) \,dK_t ^-.$$  
We define the processes:
$$\forall t\in [0,T], \ \ Y_t =u(t,W_t)\  \makebox{ and }\ Z_t =\nabla u(t,W_t).$$
Then $(Y,Z,K^+ ,K^- )$ is a solution to DRDBSDE of Corollary \ref{RBDSDE:definition}.
\end{proposition}
\begin{proof} As seen in the proof of the preceding lemma and as a consequence of Theorem 3 in \cite{MatoussiStoica}, we just need to prove the result in the case where $f=g=h=0$. So, let $u$ be a solution of 
\begin{equation}{\label{spdel0}}
 \left\{\begin{split}& d u (t,x) +
 \;  \frac{1}{2} \Delta u (t,x)dt  + \nu^+ (dt,x)-\nu^- (dt,x )=0,\\
 &\underline{v}(t,x)\leq u(t,x)\leq \overline{v}(t,x),\\
 &\int_0^T\int_{\bbR^d} \left( {u}(t,x)-\underline{v}(t,x))\right)\nu^+ (dt,dx)=\int_0^T\int_{\bbR^d} \left( \overline{v}(t,x)-{u}(t,x)\right)\nu^- (dt,dx)=0,\\
 &u_T=0.\end{split}\right. 
\end{equation}
Since $u$ is in $\mathcal{H}_T$, we can approximate it by a sequence of functions $u^n\in C_c^{\infty} ([0,T]\times\R^d)$ such that 
$$\lim_{n\rightarrow +\infty}\E\left[\sup_{t\in [0,T]}
\| u^n_t -u_t \|^2+\int_0^T \| \nabla (u^n_t -u_t)\|^2 dt\right]=0.$$
Set $Y^n_t :=u^n(t,W_t)\makebox{ and }Z^n_t :=\nabla u^n(t,W_t).$
As a consequence of It\^o's formula, we have
\begin{equation}\label{appxun} Y^n_t = u^n_0(W_0)+\int_0^t\left(  \partial_s u^n_s (W_s )+\frac12 \Delta u^n_s (W_s )\right) ds +\int_0^t Z^n_s \, dW_s.
\end{equation}
Define $K^n_t :=u^n_0 (W_0 )+\int_0^t\left(  \partial_s u^n_s (W_s )+\frac12 \Delta u^n_s (W_s )\right) ds$, { and $K_t =K^+_t -K^-_t$,} then for any $\varphi\in\mathcal{D}_T$, by integration by parts argument we obtain
\begin{eqnarray*}
\E^m \int_0^T \varphi_t (W_t)dK^n_t=-\int_0^T (u^n_s ,\partial_s\varphi_s)ds -\frac12 \int_0^T (\nabla u^n_s ,\nabla \varphi_s )\, ds.
\end{eqnarray*}
Making $n$ tend to infinity, we get, since $u$ solves \eqref{spdel0},
$$\lim_{n\rightarrow +\infty}\E^m \int_0^T \varphi_t (W_t)dK^n_t=\E^m \int_0^T \varphi_t (W_t)dK_t\,.$$
But, since $Y^n$ and $Z^n$ converge in $L^2$, we deduce that for all $t$, $K^n_t$ converges in $L^2$ to a limit which necessarily is nothing but $K_t$ as $K^n$ and $K$ belong to $\mathcal{A}_2$ $\mathbb{P}$-a.s.\\
{\red Finally coming back to equation \eqref{appxun} and making $n$ tend to $+\infty$, we conclude that $(Y,Z,K^+, K^-)$ satisfies the desired DRDBSDE.}
\end{proof}

\subsection{It\^o's formula}
In this section we will prove the It\^o's formula for the solution of DOSPDE.
Let us also remark that any solution of the nonlinear equation \eqref{SPDE1} may be viewed as 
the solution of a linear one, so we only need to establish the  It\^o's formula in the linear case i.e. for the solution of equation \eqref{OSPDElinear}.
\begin{theorem}\label{Itoformula}
Under assumptions {\bf (HD2)} and {\bf (HO)}, let $(u,\nu^+,\nu^-)$ be the solution of linear SPDE with two obstacles (\ref{OSPDElinear}) and
$\Phi:\mathbb{R}^+\times\mathbb{R}\rightarrow\mathbb{R}$ be a
function of class $\mathcal{C}^{1,2}$. We denote by $\Phi'$ and
$\Phi''$ the derivatives of $\Phi$ with respect to the space
variables and by $\frac{\partial\Phi}{\partial t}$ the partial
derivative with respect to time. We assume that there exists a constant $C>0$, such that $|\Phi''|\leq C$, $|\frac{\partial \Phi}{\partial t}|\leq C(|x|^2 \vee 1)$,
and $\Phi'(t,0)=0$ for all $t\geq0$. Then $\bbP-a.s.$ for
any $t\in[0,T]$,
\begin{equation*}\begin{split}
&\int_{\bbR^d}\Phi(t,u_t(x))dx+\frac12\int_t^T\int_{\R^d}\Phi''(s,u_s(x))\left|\nabla u_s(x)\right|^2dxds=\int_{\bbR^d}\Phi(T,\Psi(x))dx-\int_t^T\int_{\bbR^d}\frac{\partial\Phi}{\partial
s}(s,u_s(x))dxds\\&+\int_t^T(\Phi'(s,u_s),f_s)ds
-\sum_{i=1}^d\int_t^T\int_{\bbR^d}\Phi''(s,u_s(x))\partial_iu_s(x)g_i(x)dxds+\sum_{j=1}^{d^1}\int_t^T(\Phi'(s,u_s),h_j)\overleftarrow{dB}_s^j\\&
+\frac{1}{2}\sum_{j=1}^{d^1}\int_t^T\int_{\bbR^d}\Phi''(s,u_s(x))(h_{j,s}(x))^2dxds+\int_t^T\int_{\bbR^d}\Phi'(s,u_s(x))(\nu^+-\nu^-)(ds,dx).
\end{split}\end{equation*}
\end{theorem}
\begin{proof}
We begin with the penalization equation of the corresponding DRBDSDE: 
 \begin{equation*}\begin{split}
Y_{t}^n = &\,\xi +\int_{t}^{T}f_s(W_s)ds - n \int_t^T (Y_s^n  - U_s)^{+}\, ds -\frac{1}{2}\int_{t}^{T}g_s*dW_s
+ \int_{t}^{T}h_s(W_s)\cdot\overleftarrow{dB}_s \\&-\sum\limits_{i}\int_{t}^{T}Z^n_{i,s}dW_{s}^{i} + K_T^{+,n}-K_t^{+,n}\,.
\end{split}\end{equation*}
For $(Y^n,Z^n,K^{+,n})$, we have the following It\^o's formula (see Lemma \ref{itolowerobstacle}): $\forall t\in[0,T]$, $\mathbb{P}$-a.s.,
\begin{equation*}\begin{split}
&\Phi(t,Y^n_t)=\Phi(T,Y^n_T)-\int_t^T\frac{\partial\Phi}{\partial s}(s,Y^n_s)ds+\int_t^T\Phi'(s,Y^n_s)f_s(W_s)ds-\int_t^T\Phi'(s,Y^n_s)n(Y_s^n-U_s)^+ds\\
&-\frac{1}{2}\int_t^T\Phi'(s,Y^n_s)g_s*dW_s+\int_t^T\Phi'(s,Y^n_s)h_s(W_s)\cdot\overleftarrow{dB}_s-\int_t^T\Phi'(s,Y^n_s)Z^n_sdW_s+\int_t^T\Phi'(s,Y^n_s)dK_s^{+,n}\\
&+\frac{1}{2}\int_t^T\Phi''(s,Y_s^n)\left|h_s(W_s)\right|^2ds-\int_t^T\Phi''(s,Y_s^n)\langle g_s(W_s),Z^n_s\rangle ds-\frac{1}{2}\int_t^T\Phi''(s,Y_s^n)|Z^n_s|^2ds.
\end{split}\end{equation*}
It is clear  that all the terms in the above equality converge to the desired terms except those involving the process $Z^n$ and the terms $\int_0^T\Phi'(s,Y_s^n)dK_s^{+,n}$ 
and $\int_0^T\Phi'(s,Y_s^n)dK_s^{-,n}$. \\ 
For each $N\in\N$, thanks to Lemma \ref{convergence:YZK}, it is easy to verify that for example $\int_0^{t\wedge\tau_N}\Phi''(s,Y_s^n)|Z^n_s|^2ds$ converges in $L^1$ to $\int_0^{t\wedge\tau_N}\Phi''(s,Y_s)|Z_s|^2ds$, which implies the convergence almost sure by a diagonal extraction procedure of $\int_0^{t}\Phi''(s,Y_s^n)|Z^n_s|^2ds$ to $\int_0^{t}\Phi''(s,Y_s)|Z_s|^2ds$ for all $t\in [0,T]$.\\
Since
\begin{equation*}\begin{split}
&\left|\int_t^T\Phi'(s,Y_s^n)dK_s^{+,n}-\int_t^T\Phi'(s,Y_s)dK_s^+\right|\\
=&\,\left|\int_t^T\big(\Phi'(s,Y_s^n)-\Phi'(s,Y_s)\big)dK_s^{+,n}+\int_t^T\Phi'(s,Y_s)d\,(K_s^{+,n}-K_s^+)\right|.
\end{split}\end{equation*}
It is clear that 
$$\left|\int_t^T\big(\Phi'(s,Y_s^n)-\Phi'(s,Y_s)\big)dK_s^{+,n}\right|\leq  C\sup_{s\in[t,T]}\big|Y_s^n-Y_s\big|K_T^{+,n}\rightarrow0,\quad as\ \ n\rightarrow\infty.$$
As $K^{+,n}$ tends to $K^+$, we deduce that almost surely the sequence $(dK^{+,n})_n$ of measures on $[0,T]$ converges weakly to $dK^+$. Combing with the fact that  the map $s\rightarrow\Phi'(s,Y_s)$ is continuous, then we have 
$$\int_t^T\Phi'(s,Y_s)d\,(K_s^{+,n}-K_s^+)\rightarrow0,\quad as\ \ n\rightarrow\infty.$$
Similar arguments can be done for $\int_0^T\Phi'(s,Y_s^n)dK_s^{-,n}$, or more precisely, for $\int_0^T\Phi'(s,Y_s^{\delta (n)})dK_s^{-,\delta(n)}$ where $(\delta (n))_n$ is the subsequence  in the proof of Lemma \ref{convergence:YZK}. \\
At last, using the relation between $(u,\nu^+,\nu^-)$ and $(Y,Z,K^+,K^-)$, we get the desired formula. \end{proof}
\begin{remark}
We also need the It\^o's formula for the difference between two DOSPDEs which is fundamental to do the fixed point argument in the nonlinear case (see Subsection \ref{fixedpointnonlinear}) and to get the comparison theorem (see Theorem \ref{comparaison}). The proof will be similar to that of Theorem \ref{Itoformula}.
The only difference is that we begin with the It\^o's formula for the difference between the penalized solutions of two OSPDEs (see Theorem 6 in \cite{DMZ12}). 
We postpone this in the appendix. 
\end{remark}

\subsection{Proof of Theorem \ref{maintheorem} in the nonlinear case}\label{fixedpointnonlinear}
Let us consider the Picard sequence $(Y^n,Z^n)_n$
defined by $(Y^0,Z^0)=(0,0)$ and for all $n\in\mathbb{N}$ we denote by $(Y^{n+1},Z^{n+1},K^{+,n+1},K^{-,n+1})$ the solution of the linear DRBSDE as in the previous subsection
\begin{equation}\label{Appx1}\begin{split}
Y^{n+1}_t&=\xi+\int_t^Tf_s(W_s,Y^n_s,Z^n_s)ds-\frac{1}{2}\int_t^Tg_s(Y^n_s,Z^n_s)*dW_s+\int_t^Th_s(W_s,Y^n_s,Z^n_s)\cdot\overleftarrow{dB}_s
\\&\qquad-\int_t^TZ^{n+1}_sdW_s+K_T^{+,n+1}-K_t^{+,n+1}-K_T^{-,n+1}+K_t^{-,n+1}
\end{split}\end{equation}
with $L_t\leq Y^{n+1}_t\leq U_t,\; \forall \, t \in [0,T] $,  $\big(K_t^{+, n+1}\big)_{t \in[0,T]}$ and $\big(K_t^{-,n+1}\big)_{t \in[0,T]}$ being increasing continuous processes and   
\begin{equation*}
\int_{0}^{T}(Y_t^{n+1}-L_t)dK^{+,n+1}_t = \int_0^T(U_t-Y^{n+1}_t) dK_t^{-,n+1}= 0.
\end{equation*}
From now on, we introduce positive constants $\mu$ and $\varepsilon$ that we'll fix precisely later.\\
Applying It\^o's formula to $e^{\mu t}(Y^{n+1}_t-Y^n_t)^2$, we have almost surely, for all $t\in[0,T]$,
\begin{equation}\label{EgPicart}\begin{split}
&e^{\mu t}|Y^{n+1}_t-Y_t^n|^2+\int_t^Te^{\mu s}|Z^{n+1}_s-Z^n_s|^2ds+\mu \int_t^Te^{\mu s}|Y^{n+1}_s-Y^n_s|^2ds
\\=&\,2\int_t^Te^{\mu s}(Y^{n+1}_s-Y_s^n)\left[f_s(Y^n_s,Z^n_s)-f_s(Y^{n-1}_s,Z^{n-1}_s)\right]ds
\\&+\,2\int_t^Te^{\mu s}(Y^{n+1}_s-Y_s^n)d(K_s^{+,n+1}-K_s^{+,n})-2\int_t^Te^{\mu s}(Y^{n+1}_s-Y_s^n) d(K_s^{-,n+1}-K_s^{-,n})
\\&-2\,\int_t^Te^{\mu s}\langle Z^{n+1}_s-Z^n_s,g_s(Y^n_s,Z^n_s)-g_s(Y^{n-1}_s,Z^{n-1}_s)\rangle ds
\\&-\,\int_t^Te^{\mu s}(Y^{n+1}_s-Y_s^n)\left[g_s(Y^n_s,Z^n_s)-g(Y^{n-1}_s,Z^{n-1}_s)\right]*dW_s
\\&-\,2\int_t^Te^{\mu s}(Y^{n+1}_s-Y^n_s)(Z^{n+1}_s-Z^n_s)dW_s+\int_t^Te^{\mu s}|h_s(Y^n_s,Z^n_s)-h_s(Y^{n-1}_s,Z^{n-1}_s)|^2ds
\\&+\,2\int_t^Te^{\mu s}(Y^{n+1}_s-Y^n_s)\left[h_s(Y^n_s,Z^n_s)-h_s(Y^{n-1}_s,Z^{n-1}_s)\right]\cdot\overleftarrow{dB}_s\,.
\end{split}\end{equation}
Remarking the following relations: 
\begin{equation*}\begin{split}
&\int_t^Te^{\mu s}(Y^{n+1}_s-Y^{n}_s)d(K_s^{+,n+1}-K^{+,n}_s)\\
=&\int_t^Te^{\mu s}(Y_s^{n+1}-L_s)dK_s^{+,n+1}-\int_t^Te^{\mu s}(Y^n_s-L_s)dK_s^{+,n+1}
\\&-\int_t^Te^{\mu s}(Y^{n+1}_s-L_s)dK_s^{+,n}+\int_t^Te^{\mu s}(Y^n_s-L_s)dK_s^{+,n}\leq0
\end{split}\end{equation*}
and
\begin{equation*}\begin{split}
&-\int_t^Te^{\mu s}(Y^{n+1}_s-Y^n_s)d(K_s^{-,n+1}-K^{-,n}_s)\\=&\int_t^Te^{\mu s}(Y^n_s-U_s)dK_s^{-,n+1}
+\int_t^Te^{\mu s}(U_s-Y^{n+1}_s)dK_s^{-,n+1}
\\&-\int_t^Te^{\mu s}(Y^n_s-U_s)dK_s^{-,n}-\int_t^Te^{\mu s}(U_s-Y^{n+1}_s)dK_s^{-,n}\leq0.
\end{split}\end{equation*}
Let $\varepsilon\leq1$. The Lipschitz condition and Cauchy-Schwarz's inequality yield 
\begin{equation*}\begin{split}&2(Y^{n+1}_t-Y^n_t)\left[f_t(Y^n_t,Z^n_t)-f_t(Y^{n-1}_t,Z^{n-1}_t)\right]\\
\leq&\,\frac{1}{\varepsilon}\left|Y^{n+1}_t-Y^n_t\right|^2+\varepsilon\left|f_t(Y^n_t,Z^n_t)-f_t(Y^{n-1}_t,Z^{n-1}_t)\right|^2\\
\leq&\,\frac{1}{\varepsilon}\left|Y^{n+1}_t-Y^n_t\right|^2+C\varepsilon\left|Y^{n}_t-Y^{n-1}_t\right|^2+C\varepsilon\left|Z^n_t-Z^{n-1}_t\right|^2
\end{split}\end{equation*}and
\begin{equation*}\begin{split}
&2\langle Z_t^{n+1}-Z^n_t,g_t(Y^n_t,Z^n_t)-g_t(Y^{n-1}_t,Z^{n-1}_t)\rangle\\
\leq&\,2\left|Z_t^{n+1}-Z^n_t\right|(C\left|Y^n_t-Y^{n-1}_t\right|+\alpha\left|Z^n_t-Z^{n-1}_t\right|)\\
\leq&\,C\varepsilon\left|Z_t^{n+1}-Z_t^n\right|^2+\frac{C}{\varepsilon}\left|Y_t^n-Y_t^{n-1}\right|^2+\alpha\left|Z_t^{n+1}-Z_t^n\right|^2+\alpha\left|Z^n_t-Z^{n-1}_t\right|^2.
\end{split}\end{equation*}
Moreover,
\begin{equation*}\begin{split}
&\left|h_t(Y^n_t,Z^n_t)-h_t(Y^{n-1}_t,Z^{n-1}_t)\right|^2
\leq(C|Y_t^n-Y_t^{n-1}|+\beta|Z_t^n-Z_t^{n-1}|)^2\\=&\,C^2|Y_t^n-Y_t^{n-1}|^2+2C\beta|Y_t^n-Y_t^{n-1}||Z^n_t-Z^{n-1}_t|+\beta^2|Z^n_t-Z^{n-1}_t|^2\\
\leq&\, C^2(1+\frac{1}{\varepsilon})|Y_t^n-Y_t^{n-1}|^2+\beta^2(1+\varepsilon)|Z_t^n-Z_t^{n-1}|^2.
\end{split}\end{equation*}
Therefore, 
\begin{equation*}\begin{split}
&2\mathbb{E}\mathbb{E}^m\int_t^Te^{\mu s}(Y_s^{n+1}-Y_s^n)\left[f_s(Y^n_s,Z^n_s)-f_s(Y^{n-1}_s,Z^{n-1}_s)\right]ds\\
\leq&\, \frac{1}{\varepsilon}\mathbb{E}\mathbb{E}^m\int_t^Te^{\mu s}|Y^{n+1}_s-Y_s^n|^2ds+C\varepsilon \mathbb{E}\mathbb{E}^m\int_t^Te^{\mu s}\left[|Y_s^n-Y_s^{n-1}|^2+|Z_s^n-Z_s^{n-1}|^2\right]ds
\end{split}\end{equation*}and
\begin{equation*}\begin{split}
&2\mathbb{E}\mathbb{E}^m\int_t^Te^{\mu s}\langle Z^{n+1}_s-Z_s^n,g_s(Y^n_s,Z^n_s)-g_s(Y^{n-1}_s,Z^{n-1}_s)\rangle ds\\
\leq& (C\varepsilon+\alpha)\mathbb{E}\mathbb{E}^m\int_t^Te^{\mu s}\left|Z^{n+1}_s-Z_s^n\right|^2ds
+\frac{C}{\varepsilon}\mathbb{E}\mathbb{E}^m\int_t^Te^{\mu s}\left|Y_s^n-Y_s^{n-1}\right|^2ds\\&+\alpha \mathbb{E}\mathbb{E}^m\int_t^Te^{\mu s}\left|Z_s^n-Z_s^{n-1}\right|^2ds.
\end{split}\end{equation*}
Also, we get 
\begin{equation*}\begin{split}
&\mathbb{E}\mathbb{E}^m\int_t^Te^{\mu s}|h_s(Y^n_s,Z^n_s)-h_s(Y^{n-1}_s,Z^{n-1}_s))|^2ds\\
\leq&\,C^2(1+\frac{1}{\varepsilon})\mathbb{E}\mathbb{E}^m\int_t^Te^{\mu s}|Y_s^n-Y_s^{n-1}|^2ds
+\beta^2(1+\varepsilon)\mathbb{E}\mathbb{E}^m\int_t^Te^{\mu s}|Z_s^n-Z_s^{n-1}|^2ds.
\end{split}\end{equation*}
We deduce that 
\begin{equation*}\begin{split}
&(\mu-\frac{1}{\varepsilon})\mathbb{E}\mathbb{E}^m\int_t^Te^{\mu s}|Y^{n+1}_s-Y_s^n|^2ds+(1-\alpha-C\varepsilon)\mathbb{E}\mathbb{E}^m\int_t^Te^{\mu s}|Z^{n+1}_s-Z_s^n|^2ds\\
\leq&\,C(C+1)(1+\frac{1}{\varepsilon})\mathbb{E}\mathbb{E}^m\int_t^Te^{\mu s}|Y_s^n-Y_s^{n-1}|^2ds+(C\varepsilon+\alpha+\beta^2(1+\varepsilon))\mathbb{E}\mathbb{E}^m\int_t^Te^{\mu s}|Z_s^n-Z_s^{n-1}|^2ds.
\end{split}\end{equation*}
We take the norm $$\left\|(Y,Z)\right\|^2_{\mu,\delta}:=\mathbb{E}\mathbb{E}^m\int_0^Te^{\mu s}(\delta\left|Y_t\right|^2+\left|Z_t\right|^2)dt\,.$$
We can choose $\varepsilon$ small enough and then $\mu$ such that 
$$C\varepsilon+\alpha+\beta^2(1+\varepsilon)<1-\alpha-C\varepsilon\quad\mbox{and}\quad \frac{\mu-1/\varepsilon}{1-\alpha-C\varepsilon}=\frac{C(C+1)(1+1/\varepsilon)}{C\varepsilon+\alpha+\beta^2(1+\varepsilon)}.$$
If we set $\delta=\frac{\mu-1/\varepsilon}{1-\alpha-C\varepsilon}$ and $\delta_0=\frac{C\varepsilon+\alpha+\beta^2(1+\varepsilon)}{1-\alpha-C\varepsilon}\in(0,1)$, we have the following inequality:
\begin{equation*}\begin{split}
\left\|(Y^{n+1}-Y^n,Z^{n+1}-Z^n)\right\|^2_{\mu,\delta}&\leq\delta_0\left\|(Y^n-Y^{n-1},Z^n-Z^{n-1})\right\|^2_{\mu,\delta}\\&\leq...\\&\leq\delta_0^n\left\|(Y^1,Z^1)\right\|^2_{\mu,\delta}.\end{split}\end{equation*}
Since $\delta_0^n\rightarrow0$
when $n\rightarrow\infty$, we conclude that $(Y^n,Z^n)$ is a Cauchy sequence in the $L^2$-space hence converges to a couple $(Y,Z)$ w.r.t the norm $\|\cdot\|_{\mu,\delta}$. \\
Now, coming back to equality \eqref{EgPicart},  similar calculations to the previous ones plus Burkholder-Davies-Gundy's inequality yield
\begin{equation*}\begin{split}
\E\E^m\left[\sup_{t\in [0,T]}|Y^{n+1}_t -Y^n_t|^2\right]\leq&\, C\E\E^m \left[ \int_0^T |Y^{n}_s -Y^{n-1}_s|^2 +|Z^{n}_s -Z^{n-1}_s|^2\, ds\right]\\
&+C\E\E^m \left[\left(\int_0^T|Y^{n+1}_s-Y_s^n|^2 \left| g_s(Y^n_s,Z^n_s)-g(Y^{n-1}_s,Z^{n-1}_s)\right|^2 \, ds\right)^{1/2}\right]\\
&+C\E\E^m \left[\left(\int_0^T|Y^{n+1}_s-Y_s^n|^2 \left| h_s(Y^n_s,Z^n_s)-h(Y^{n-1}_s,Z^{n-1}_s)\right|^2 \, ds\right)^{1/2}\right]\\
&+C\E\E^m \left[\left(\int_0^T|Y^{n+1}_s-Y_s^n|^2 \left|Z^{n+1}_s-Z^n_s\right|^2 \, ds\right)^{1/2}\right]
\end{split}\end{equation*}
and then we remark that
\begin{equation*}\begin{split}
&\E\E^m \left[\left(\int_0^T|Y^{n+1}_s-Y_s^n|^2 \left| g_s(Y^n_s,Z^n_s)-g(Y^{n-1}_s,Z^{n-1}_s)\right|^2 \, ds\right)^{1/2}\right] \\ 
\leq&\,\E\E^m \left[\sup_{t\in [0,T]}|Y^{n+1}_t-Y_t^n| \left(\int_0^T\left| g_s(Y^n_s,Z^n_s)-g(Y^{n-1}_s,Z^{n-1}_s)\right|^2 \, ds\right)^{1/2}\right]\\
\leq&\,\varepsilon' \E\E^m \left[\sup_{t\in [0,T]}|Y^{n+1}_t-Y_t^n|^2\right]+\displaystyle\frac{1}{4\varepsilon'}\E\E^m \left[ \int_0^T\left| g_s(Y^n_s,Z^n_s)-g(Y^{n-1}_s,Z^{n-1}_s)\right|^2 \, ds\right]\\
\leq&\,\varepsilon' \E\E^m \left[\sup_{t\in [0,T]}|Y^{n+1}_t-Y_t^n|^2\right]+ C\E\E^m \left[ \int_0^T |Y^{n}_s -Y^{n-1}_s|^2 +|Z^{n}_s -Z^{n-1}_s|^2\, ds\right],
\end{split}\end{equation*}
where $\varepsilon'$ is arbitrary small. We do the same trick for the {\red last two terms} and finally obtain the following estimate:
\begin{equation*}\begin{split}&\E\E^m\left [\sup_{t\in [0,T]}|Y^{n+1}_t -Y^n_t|^2\right]\\\leq&\,C \left\| (Y^n -Y^{n-1}, Z^n-Z^{n-1})\right\|_{\mu,\delta}^2+C \left\| (Y^{n+1} -Y^{n}, Z^{n+1}-Z^{n})\right\|_{\mu,\delta}^2\leq C\delta_0^{n-1}\,.\end{split}\end{equation*}
So, by standard arguments, we obtain the following convergence:
$$\lim_{n\rightarrow +\infty }\E\E^m \left[ \sup_{t\in [0,T]}|Y^n_t -Y_t|^2 +\int_0^T |Z^n_t -Z_t|^2\, dt\right]=0,$$
here again this ensures that we can choose for $Y$ a time-continuous version.\\
It now remains to  prove the convergences of $K^{+,n}$ and $K^{-,n}$, to the end {\red{we recall the function $\psi$ introduced in Section \ref{section:penalization} which "separates" the two obstacles and which  is defined as a function $\psi\in C^2$ satisfying  $\psi(x)=x$ when $x\in(-\infty, -\kappa]$ and  $\psi(x)=0$ when $x\in[-\frac{\kappa}{2},+\infty)$}}. We have almost surely for $n,m\in\N$ and $\forall t\in[0,T]$, 
\begin{equation*}
\begin{split}
&\psi(Y_t^n-\tilde Y_t)-\psi(Y_t^m-\tilde Y_t)\\=&\,\int_t^T\psi'(Y_s^n-\tilde Y_s)dK_s^{+,n}-\int_t^T\psi'(Y_s^m-\tilde Y_s)dK_s^{+,m}-\int_t^T\psi'(Y_s^n-\tilde Y_s)dK^{-,n}_s\\&+\int_t^T\psi'(Y_s^m-\tilde Y_s)dK^{-,m}_s-\int_t^T\psi'(Y_s^n-\tilde Y_s)(Z_s^n-\tilde Z_s)dW_s+\int_t^T\psi'(Y_s^m-\tilde Y_s)(Z_s^m-\tilde Z_s)dW_s\\&-\frac{1}{2}\int_t^T\psi''(Y_s^n-\tilde Y_s)|Z_s^n-\tilde Z_s|^2ds+\frac{1}{2}\int_t^T\psi''(Y_s^m-\tilde Y_s)|Z_s^m-\tilde Z_s|^2ds\\&+\int_t^T\psi'(Y_s^n-\tilde Y_s)(f_s(Y_s^{n-1},Z_s^{n-1})-\tilde f_s)ds-\int_t^T\psi'(Y_s^m-\tilde Y_s)(f_s(Y_s^{m-1},Z_s^{m-1})-\tilde f_s)ds\\&-\frac{1}{2}\int_t^T\psi'(Y_s^n-\tilde Y_s)(g_s(Y_s^{n-1},Z_s^{n-1})-\tilde g_s)*dW_s+\frac{1}{2}\int_t^T\psi'(Y_s^m-\tilde Y_s)(g_s(Y_s^{m-1},Z_s^{m-1})-\tilde g_s)*dW_s\\&+\int_t^T\psi'(Y_s^n-\tilde Y_s)(h_s(Y_s^{n-1},Z_s^{n-1})-\tilde h_s)\cdot\overleftarrow{dB}_s-\int_t^T\psi'(Y_s^m-\tilde Y_s)(h_s(Y_s^{m-1},Z_s^{m-1})-\tilde h_s)\cdot\overleftarrow{dB}_s\\&+\frac{1}{2}\int_t^T\psi''(Y_s^n-\tilde Y_s)|h_s(Y_s^{n-1},Z_s^{n-1})-\tilde h_s|^2ds-\frac{1}{2}\int_t^T\psi''(Y_s^m-\tilde Y_s)|h_s(Y_s^{m-1},Z_s^{m-1})-\tilde h_s|^2ds\\
&-\int_t^T\psi''(Y_s^n-\tilde Y_s)\langle g_s(Y_s^{n-1},Z_s^{n-1})-\tilde g_s,Z_s^n-\tilde Z_s\rangle ds+\int_t^T\psi''(Y_s^m-\tilde Y_s)\langle g_s(Y_s^{m-1},Z_s^{m-1})-\tilde g_s,Z_s^m-\tilde Z_s\rangle ds.
\end{split}
\end{equation*}
{\red{Noting that by the strict separability condition \textbf{(HO)-(iv)} and the structure of $ \psi$, we get} }
\begin{equation*}\begin{split}&\int_t^T\psi'(Y_s^n-\tilde Y_s)dK_s^{+,n}=\int_t^T\psi'(L_s-\tilde Y_s)dK_s^{+,n}=K_T^{+,n}-K_t^{+,n},\\& \int_t^T\psi'(Y_s^m-\tilde Y_s)dK_s^{+,m}=\int_t^T\psi'(L_s-\tilde Y_s)dK_s^{+,m}=K_T^{+,m}-K_t^{+,m},\end{split}\end{equation*}
and \begin{equation*}\begin{split}&\int_t^T\psi'(Y_s^n-\tilde Y_s)dK^{-,n}_s=\int_t^T\psi'(U_s-\tilde Y_s)dK^{-,n}_s=0\,,
\\&\int_t^T\psi'(Y_s^m-\tilde Y_s)dK^{-,m}_s=\int_t^T\psi'(U_s-\tilde Y_s)dK^{-,m}_s=0\,.\end{split}\end{equation*}
Therefore,
\begin{equation*}
\begin{split}
&|K_T^{+,n}-K_t^{+,n}-(K_T^{+,m}-K_t^{+,m})|\\\leq&\,|\psi(Y_t^n-\tilde Y_t)-\psi(Y_t^m-\tilde Y_t)|+\left|\int_t^T\big(\psi'(Y_s^n-\tilde Y_s)(Z_s^n-\tilde Z_s)-\psi'(Y_s^m-\tilde Y_s)(Z_s^m-\tilde Z_s)\big)dW_s\right|\\&+\frac{1}{2}\int_t^T\big|\psi''(Y_s^n-\tilde Y_s)|Z_s^n-\tilde Z_s|^2-\psi''(Y_s^m-\tilde Y_s)|Z_s^m-\tilde Z_s|^2\big|ds\\&+\int_t^T\left|\psi'(Y_s^n-\tilde Y_s)(f_s(Y_s^{n-1},Z_s^{n-1})-\tilde f_s)-\psi'(Y_s^m-\tilde Y_s)(f_s(Y_s^{m-1},Z_s^{m-1})-\tilde f_s)\right|ds\\&+\frac{1}{2}\left|\int_t^T\left(\psi'(Y_s^n-\tilde Y_s)(g_s(Y_s^{n-1},Z_s^{n-1})-\tilde g_s)-\psi'(Y_s^m-\tilde Y_s)(g_s(Y_s^{m-1},Z_s^{m-1})-\tilde g_s)\right)*dW_s\right|\\&+\left|\int_t^T\left(\psi'(Y_s^n-\tilde Y_s)(h_s(Y_s^{n-1},Z_s^{n-1})-\tilde h_s)-\psi'(Y_s^m-\tilde Y_s)(h_s(Y_s^{m-1},Z_s^{m-1})-\tilde h_s)\right)\cdot\overleftarrow{dB}_s\right|\\&+\frac{1}{2}\int_t^T\left|\psi''(Y_s^n-\tilde Y_s)|h_s(Y_s^{n-1},Z_s^{n-1})-\tilde h_s|^2-\psi''(Y_s^m-\tilde Y_s)|h_s(Y_s^{m-1},Z_s^{m-1})-\tilde h_s|^2\right|ds\\&+\frac{1}{2}\int_t^T\left|\psi''(Y_s^n-\tilde Y_s)\langle g_s(Y_s^{n-1},Z_s^{n-1})-\tilde g_s,Z_s^n-\tilde Z_s\rangle-\psi''(Y_s^m-\tilde Y_s)\langle g_s(Y_s^{m-1},Z_s^{m-1})-\tilde g_s,Z_s^m-\tilde Z_s\rangle\right|ds,
\end{split}
\end{equation*}
then, taking the supremum in $t$ and the expectation, thanks to the Burkholder-Davies-Gundy inequality, we obtain 
\begin{equation*}
\begin{split}
&\E\E^m\left [\sup_{t\in [0,T]}\left |K_T^{+,n}-K_t^{+,n}-(K_T^{+,m}-K_t^{+,m})\right|\right]\\\leq&\,\E\E^m \left[\sup_{t\in[0,T]}\left |\psi(Y_t^n-\tilde Y_t)-\psi(Y_t^m-\tilde Y_t)\right|+\left(\int_0^T\left|\psi'(Y_s^n-\tilde Y_s)(Z_s^n-\tilde Z_s)-\psi'(Y_s^m-\tilde Z_s)(Z_s^m-\tilde Z_s)\right|^2ds\right)^{1/2}\right]\\&\ +\E\E^m \left[\int_0^T\left|\psi''(Y_s^n-\tilde Y_s)|Z_s^n-\tilde Z_s|^2-\psi''(Y_s^m-\tilde Y_s)|Z_s^m-\tilde Z_s|^2\right|ds\right]\\&\ +\E\E^m \left[\int_0^T\left|\psi'(Y_s^n-\tilde Y_s)(f_s(Y_s^{n-1},Z_s^{n-1})-\tilde f_s)-\psi'(Y_s^m-\tilde Y_s)(f_s(Y_s^{m-1},Z_s^{m-1})-\tilde f_s)\right|ds\right]\\&\  +\E\E^m \left[\left(\int_0^T\left(\psi'(Y_s^n-\tilde Y_s)(g_s(Y_s^{n-1},Z_s^{n-1})-\tilde g_s)-\psi'(Y_s^m-\tilde Y_s)(g_s(Y_s^{m-1},Z_s^{m-1})-\tilde g_s)\right)^2\, ds \right)^{1/2}\right]\\&\ +\E\E^m \left[\left(\int_0^T\left|\psi'(Y_s^n)h_s(Y_s^{n-1},Z_s^{n-1})-\psi'(Y_s^m)h_s(Y_s^{m-1},Z_s^{m-1})\right|^2\, ds\right)^{1/2}\right]\\&\ +\E\E^m\left[\int_0^T\left|\psi''(Y_s^n-\tilde Y_s)\left|h_s(Y_s^{n-1},Z_s^{n-1})-\tilde h_s\right|^2-\psi''(Y_s^m-\tilde Y_s)\left|h_s(Y_s^{m-1},Z_s^{m-1})-\tilde h_s\right|^2\right|ds\right]\\&\ +\E\E^m\left[\int_0^T\left|\psi''(Y_s^n-\tilde Y_s)\left\langle g_s(Y_s^{n-1},Z_s^{n-1})-\tilde g_s,Z_s^n-\tilde Z_s\right\rangle-\psi''(Y_s^m-\tilde Y_s)\left\langle g_s(Y_s^{m-1},Z_s^{m-1})-\tilde g_s,Z_s^m-\tilde Z_s\right\rangle\right|ds\right].
\end{split}
\end{equation*}
By extracting a subsequence if necessary, we can assume that $\sup_{t\in [0,T]}|Y^n_t -Y_t|$ tends to $0$ almost-everywhere, this ensures that all the terms on the right hand side of the previous inequality tend to $0$ as $n,m\rightarrow +\infty$. To see it, let us study  the second term: 
\begin{equation*}\begin{split}
&\E\E^m \left[\int_0^T\left|\psi''(Y_s^n-\tilde Y_s)|Z_s^n-\tilde Z_s|^2-\psi''(Y_s-\tilde Y_s)|Z_s-\tilde Z_s|^2\right|ds\right]\\\leq&\ \E\E^m \left[\int_0^T\left|\psi''(Y_s^n-\tilde Y_s)\right|\left||Z_s^n-\tilde Z_s|^2-|Z_s-\tilde Z_s|^2\right|\,ds\right]\\&\,+\E\E^m \left[ \int_0^T\left |\psi''(Y_s-\tilde Y_s)-\psi''(Y_s^n-\tilde Y_s)\right|\left|Z_s-\tilde Z_s\right|^2 \, ds\right].
\end{split}\end{equation*}
The first term of the right member tends to $0$ since $\psi''$ is bounded and the second by use of the dominated convergence theorem. Repeating these kinds of arguments, we get that, for a subsequence,  $K^{+,n}$ converges uniformly on $t$ in $L^1$ to an increasing continuous process $K^+$.   In the same way we have the convergence of $K^{-,n}$ to an increasing continuous process $K^-$.\\
The fact that $K^+$ and $K^-$ satisfy the minimal Skohorod condition can be proven as in the proof of Lemma \ref{convergence:YZK}.
Now passing to the limit in \eqref{Appx1} for a well-chosen subsequence we get that $(Y,Z,K^+ ,K^-)$ solves the DRBSDE
\begin{equation*}\begin{split}
Y_t&=\xi+\int_t^Tf_s(W_s,Y_s,Z_s)ds-\frac{1}{2}\int_t^Tg_s(Y_s,Z_s)*dW_s+\int_t^Th_s(W_s,Y_s,Z_s)\cdot\overleftarrow{dB}_s
\\&\quad-\int_t^TZ_sdW_s+K_T^{+}-K_t^{+}-K_T^{-}+K_t^{-}\,.
\end{split}\end{equation*}
We end this proof by establishing that this solution may be viewed as the solution of a linear DRBSDE, so that all the results of the previous section apply. More precisely, let $(\bY ,\bZ ,\bK^+ ,\bK^- )$ be a solution of the linear DRBSDE (with same obstacles $U$ and $L$):
\begin{equation*}\begin{split}
\bY_t&=\xi+\int_t^Tf_s(W_s,Y_s,Z_s)ds-\frac{1}{2}\int_t^Tg_s(Y_s,Z_s)*dW_s+\int_t^Th_s(W_s,Y_s,Z_s)\cdot\overleftarrow{dB}_s\\&\quad-\int_t^T\bZ_sdW_s
+\bK_T^{+}-\bK_t^{+}-\bK_T^{-}+\bK_t^{-}\,.\end{split}\end{equation*}
Then, $Y_t -\bY_t =-\int_t^T (Z_s -\bZ_s)dW_s + (K_T^+ -\bK_T^+)-(K_t^+ -\bK_t^+)-(K_T^- -\bK_T^-)+(K_t^- -\bK_t^-)$ hence is a $\mathcal{G}_t$-semi-martingale. Now applying the It\^o-Tanaka formula to $((Y_t -\bY_t)^+)^2$, we obtain by similar arguments to those used in the proof of Theorem 1.3 in \cite{Hamadene-Hasani} that $Y=\bY$, hence as a consequence of Doob-Meyer's theorem, $Z=\bZ$ and $K^+-K^- =\bK^+-\bK^-$. Applying one more time It\^o's formula to $\psi (Y_t -\tilde{Y}_t)=\psi (\bY_t -\tilde{Y}_t)$, we immediately get $K^+=\bK^+$ and $K^- =\bK^-$. $\hspace{5.2cm}\Box$\\

Then we can do a similar argument as in the proof of Theorem \ref{maintheorem} in the linear case to get the result on $u$. Precisely, the above proof provides that the Picard sequence $(Y^n, Z^n)$ is a Cauchy sequence, then using the relation between $(u^n, \nabla u^n)$ and $(Y^n, Z^n)$, we obtain that the corresponding Picard sequence $u^n$ is a Cauchy sequence in $\mathcal{H}_T$ and hence has a limit $u$ in this space.

\subsection{Comparison theorem}
We can also establish the comparison theorem for the solution of our two-obstacle problem. 
\begin{theorem}
\label{comparaison}
Let $\Psi ^2,f^2,\overline{v}^2,\underline{v}^2$ be similar to $\Psi^1,f^1,\overline{v}^1,\underline{v}^1$ and
let $\left( u^1,\nu^{1,+},\nu^{1,-} \right) $ be the solution of the two-obstacle problem
corresponding to $\left( \Psi^1,f^1,g,h,\overline{v}^1,\underline{v}^1\right) $ and $\left( u^2,\nu
^{2,+},\nu^{2,-}\right)$ be the solution corresponding to $\left( \Psi^{2
},f^{2},g,h,\overline{v}^{2},\underline{v}^{2}\right) .$ Assume that the following conditions
hold

\begin{description}
\item[$\left( i\right) $] $\Psi^1 \leq \Psi ^2,\ \ dx\otimes d\mathbb{P}$
-a.e.,

\item[$\left( ii\right) $] $f^1\left( u^1,\nabla u^1\right) \leq f^2\left(
u^1,\nabla u^1\right) ,\ \ dtdx\otimes \mathbb{P}$ -a.e.,

\item[$\left( iii\right) $] $\overline{v}^1\leq \overline{v}^{2}\ and\ \ \underline{v}^1\leq\underline{v}^{2},\ \ dtdx\otimes \mathbb{P}$ -a.e..
\end{description}

Then one has $u^1\leq u^{2},\ \ dtdx\otimes \mathbb{P}-$a.e..
\end{theorem}
\begin{proof}
We put $\hat{u}=u^1-u^2$, $\hat{\Psi}=\Psi^1-\Psi^2$, $\hat{f}_t=f^1(t,u^1_t,\nabla
u^1_t)-f^2(t,u^2_t,\nabla u^2_t)$, $\hat{g}_t=g(t,u^1_t,\nabla
u^1_t)-g(t,u^2_t,\nabla u^2_t)$ and $\hat{h}_t=h(t,u^1_t,\nabla
u_t)-h(t,u^2_t,\nabla u^2_t)$.

One starts with the following version of It\^o's formula (see Lemma \ref{itolowerobstacle} and Remark \ref{remarkitoforpositivepart}), written with some
quasicontinuous versions $\tilde{u}^1,\tilde{u}^2$ of the
solutions $u^1,u^{2}$ in the terms involving the regular measures $\nu^{1,+},\nu^{1,-}
,\nu ^{2,+ },\nu^{2,-},$
\begin{equation*}
\begin{split}
\mathbb{E}\left\|\hat{u}_t^{+}\right\|
^{2}+\mathbb{E}\int_{t}^{T} \big\|\nabla\hat{u}_s^{+} \big\|^2 \, ds
&=\mathbb{E}\big\|\hat{\Psi}^{+}\big\|^{2}
-2\mathbb{E}\int_{t}^{T}\big( \hat{u}_s^{+}, \hat{f}_s\big) ds+2\mathbb{E}\int_{t}^{T}\left(\nabla\hat{u}_s^+,\hat{g}_s\right)ds\\
&+\mathbb{E}\int_{t}^{T}\big\|\1_{\{\hat{u}_s>0\}}\big|\hat{h}_{s}\big|\big\|^{2}\,ds
+2\mathbb{E}\int_{t}^{T}\int_{\mathbb{R}^{d}}\hat{u}_s^{+}\left( x\right)
\left( \nu^{1,+} -\nu ^{2,+ }\right) \left(ds,dx\right)\\&-2\mathbb{E}\int_{t}^{T}\int_{\mathbb{R}^{d}}\hat{u}_s^{+}\left( x\right)
\left( \nu^{1,-} -\nu ^{2,- }\right) \left(ds,dx\right).
\end{split}
\end{equation*}
Remark that on $\{u^1\leq u^2\}$, $(u^1-u^2)^+=0$ and on $\{u^1>u^2\}$, $\nu^{1,+}(ds,dx)=0$, 
then
\begin{equation*}
2\mathbb{E}\int_{t}^{T}\int_{\mathbb{R}^{d}}\hat{u}_s^{+}\left( x\right)
\left( \nu^{1,+} -\nu ^{2,+ }\right) \left(ds,dx\right)\leq 0.
\end{equation*}
Similarly, on $\{u^1\leq u^2\}$, $(u^1-u^2)^+=0$ and on $\{u^1>u^2\}$, $\nu^{2,-}(ds,dx)=0$, 
then
\begin{equation*}
2\mathbb{E}\int_{t}^{T}\int_{\mathbb{R}^{d}}\hat{u}_s^{+}\left( x\right)
\left( \nu^{1,-} -\nu ^{2,- }\right) \left(ds,dx\right)\geq 0.
\end{equation*}
And then one concludes the proof by Gronwall's lemma.
\end{proof}

\section{Appendix}
The aim of this Appendix is to prove the It\^o's formula in the one-obstacle case. To this end, we are given $\xi\in L^2 (\R^d )$ and predictable (linear) coefficients 
$f=f^0$, $g=g^0$, $h=h^0$ satisfying Assumption {\bf (HD2)}.
\begin{lemma}\label{itolowerobstacle}
Let $\Phi$ be the function satisfying the conditions in Theorem \ref{Itoformula} and $(Y,Z,K)$ be the solution of the lower obstacle problem for BDSDE: 
\begin{equation}\label{BDSDElowerobstacle}\left\{\begin{split}
Y_t&=\xi+\int_t^Tf_sds-\frac{1}{2}\int_t^Tg_s*dW_s+\int_t^Th_s\cdot\overleftarrow{dB}_s-\int_t^TZ_sdW_s+K_T-K_t\,,\\
Y_t &\geq L_t\,,\\
\int_0^T& (Y_t -L_t )dK_t =0\,.
\end{split}\right.\end{equation}
Then, the following It\^o's formula holds almost surely, for any $t\in[0,T]$, 
\begin{equation*}\begin{split}
\Phi(t,Y_t)=&\,\Phi(T,Y_T)-\int_t^T\frac{\partial\Phi}{\partial s}(s,Y_s)ds+\int_t^T\Phi'(s,Y_s)f_sds-\frac{1}{2}\int_t^T\Phi'(s,Y_s)g_s*dW_s\\
&+\int_t^T\Phi'(s,Y_s)h_s\cdot\overleftarrow{dB}_s-\int_t^T\Phi'(s,Y_s)Z_sdW_s+\frac{1}{2}\int_t^T\Phi''(s,Y_s)|h_s|^2ds\\
&-\int_t^T\Phi''(s,Y_s)\langle g_s,Z_s\rangle ds-\frac{1}{2}\int_t^T\Phi''(s,Y_s)|Z_s|^2ds+\int_t^T\Phi'(s,Y_s)dK_s\,.
\end{split}\end{equation*}
\end{lemma}
\begin{proof}
We consider the following penalization equation
\begin{equation}
Y_t^n=\xi+\int_t^Tf_sds-\frac{1}{2}\int_t^Tg_s*dW_s+\int_t^Th_s\cdot\overleftarrow{dB}_s-\int_t^TZ_{i,s}^ndW_s^i+\int_t^Tn(Y_s^n-L_s)^-ds.
\end{equation}
Using the same arguments as in the proof of  Lemma 4.3 in \cite{Stoica} and the It\^o formula for doubly stochastic It\^o processes, see Lemma 1.3 in \cite{PardouxPeng94}, we get that, for all $t\in[0,T]$, almost surely, 
\begin{equation*}\begin{split}
\Phi(t,Y^n_t)=&\,\Phi(T,Y^n_T)-\int_t^T\frac{\partial\Phi}{\partial s}(s,Y^n_s)ds+\int_t^T\Phi'(s,Y^n_s)f_sds-\frac{1}{2}\int_t^T\Phi'(s,Y^n_s)g_s*dW_s\\
&+\int_t^T\Phi'(s,Y^n_s)h_s\cdot\overleftarrow{dB}_s-\int_t^T\Phi'(s,Y^n_s)Z^n_sdW_s+\frac{1}{2}\int_t^T\Phi''(s,Y_s^n)|h_s|^2ds\\
&-\int_t^T\Phi''(s,Y_s^n)\langle g_s,Z^n_s\rangle ds-\frac{1}{2}\int_t^T\Phi''(s,Y_s^n)|Z^n_s|^2ds+\int_t^T\Phi'(s,Y^n_s)n(Y^n_s-L_s)^-ds.
\end{split}\end{equation*}
From \cite{MatoussiStoica}, we know that the triple $(Y^n,Z^n,K^n)$ strongly converges to $(Y,Z,K)$ which is the solution of the lower obstacle problem for SPDE \eqref{SPDE1}. Hence, all the terms in the above equality converge. We get the desired formula by taking limits. 
\end{proof}

\begin{lemma}\label{comparisonthmlinearRBDSDE}(Comparison theorem for the linear reflected BDSDEs)
Let  $\xi\in L^2 (\R^d )$ and predictable  coefficients 
$f$, $g$, $h$ satisfying Assumption {\bf (HD2)}. Let $(Y,Z,K)$ be the solution of the reflected BDSDEs \eqref{BDSDElowerobstacle}. Let $\xi'\in L^2 (\R^d )$ and $f'$ another predictable coefficient satisfying {\bf (HD2)}. Let  $(Y',Z',K')$ be the solution of the reflected BDSDEs with coefficients $f'$, $g$, $h$, terminal value $\xi'$ and same lower obstacle $L$. 
If \begin{enumerate}
      \item $\xi\leq\xi', \ \mathbb{P}-a.s.$, 
      \item $f\leq f',\ dt\otimes\mathbb{P}-a.e.$.
    \end{enumerate}
Then we have $\mathbb{P}$-almost surely, $Y_t\leq Y'_t$ for all $t\in [0,T]$ and $dK_t\geq dK'_t$.
\end{lemma}

\begin{proof}
We consider the following two penalized equations:
\begin{equation*}
Y_t^n=\xi+\int_t^Tf_s(W_s)ds-\frac{1}{2}\int_t^Tg_s*dW_s+\int_t^Th_s(W_s)\cdot\overleftarrow{dB}_s-\sum_i\int_t^TZ_{i,s}^ndW_s^i-n\int_t^T(Y_s^n-L_s)^-ds,
\end{equation*}
\begin{equation*}
Y_t^{'n}=\xi'+\int_t^Tf'_s(W_s)ds-\frac{1}{2}\int_t^Tg_s*dW_s+\int_t^Th_s(W_s)\cdot\overleftarrow{dB}_s-\sum_i\int_t^TZ_{i,s}^{'n}dW_s^i-n\int_t^T(Y_s^{'n}-L_s)^-ds.
\end{equation*}
We denote
$$F_t(Y_t^n)=f_t-n(Y_t^n-L_t)^-\quad\mbox{and}\quad F'_t(Y_t^n)=f'_t-n(Y_t^n-L_t)^-,$$
due to assumption 2, we have that $F_t(Y^n_t)\leq F'_t(Y_t^n),\
dt\otimes \mathbb{P}-a.e.$. Therefore, applying It\^o's formula to $\big( (Y_t^n -Y_t^{'n})^+\big)^2 $ and standard arguments as the comparison theorem
for BSDEs (non-reflected), we get that $\forall
t\in[0,T]$, $Y_t^n\leq Y_t^{'n},\ \mathbb{P}-a.s.$, thus
$n(Y_t^n-L_t)^-\geq n(Y_t^{'n}-L_t)^-$, which implies by passing to the limit that 
  $dK_t\geq dK'_t$ for any $t\in[0,T]$.
\end{proof}

\vspace{2mm}
Next we prove the It\^o's formula for the difference between the solutions of two DOSPDEs.\\
We still consider $(u,\nu^+,\nu^-)$ the solution of linear equation as in Subsection \ref{section:penalization}
\begin{equation*}
 \left\{\begin{split}& du(t,x) + \frac{1}{2} \Delta u (t,x)dt  + f(t,x)dt+ \mbox{div} g(t,x)dt +h(t,x)\cdot\overleftarrow{dB}_t +\nu^+ (dt,x)-\nu^- (dt,x)=0,\\
 &\underline{v}(t,x)\leq u(t,x)\leq \overline{v}(t,x),
 \end{split}\right. 
\end{equation*}
and consider another linear equation with adapted coefficients $\bar{f},\bar{g},\bar{h}$ respectively in $L^2 (\Omega\times[0,T]\times\mathbb{R}^d;\mathbb{R})$,  $L^2 (\Omega\times[0,T]\times\mathbb{R}^d;\mathbb{R}^d)$ and $L^2 (\Omega\times[0,T]\times\mathbb{R}^d;\mathbb{R}^{d^1})$ and the obstacles $\underline{o}$ and $\overline{o}$ satisfying Assumption ${\bf (HO)}$. We denote by $(y, \bar{\nu}^+,\bar{\nu}^-)$ the unique solution to the associated DOSPDE with terminal condition $y_T =u_T =\Psi$:
\begin{equation*}
 \left\{\begin{split}& dy(t,x) + \frac{1}{2} \Delta y (t,x)dt  + \bar{f}(t,x)dt+ \mbox{div} \bar{g}(t,x)dt + \bar{h}(t,x)\cdot\overleftarrow{dB}_t +\bar{\nu}^+ (dt,x)-\bar{\nu}^- (dt,x)=0,\\
 &\underline{o}(t,x)\leq y(t,x)\leq \overline{o}(t,x).
 \end{split}\right. 
\end{equation*}

\begin{lemma}\label{itofordifference}
Let $\Phi$ as in Theorem \ref{Itoformula}, then the difference of the two solutions satisfy the following It\^o's formula: $\forall t\in[0,T]$, $\mathbb{P}-$a.s.,
\begin{equation}
\begin{split}
&\int_{\bbR^d}\Phi(t,u_t(x)-y_t(x))dx+\frac12 \int_t^T \Phi''(s,u_s-y_s)|\nabla (u_s-y_s)|^2 ds=-\int_t^T\int_{\bbR^d}\frac{\partial\Phi}{\partial s}(s,u_s-y_s)dxds
\\&+\int_t^T(\Phi'(s,u_s-y_s),f_s-\bar{f}_s)ds-\sum_{i=1}^d\int_t^T\int_{\bbR^d}\Phi''(s,u_s-y_s)\partial_i(u_s-y_s)(g^i_s-\bar{g}^i_s)dxds
\\&+\sum_{j=1}^{d^1}\int_t^T(\Phi'(s,u_s-y_s),h^j_s-\bar{h}^j_s)\overleftarrow{dB}_s^j+\frac{1}{2}\sum_{j=1}^{d^1}\int_t^T\int_{\bbR^d}\Phi''(s,u_s-y_s)(h^j_s-\bar{h}^j_s)^2dxds
\\&+\int_t^T\int_{\bbR^d}\Phi'(s,\tilde{u}_s-\tilde{y}_s)(\nu^+-\bar{\nu}^+)(ds, dx)-\int_t^T\int_{\bbR^d}\Phi'(s,\tilde{u}_s-\tilde{y}_s)(\nu^--\bar{\nu}^-)(ds, dx).
\end{split}
\end{equation}
\end{lemma}
\begin{proof}
We begin with the penalization equations of the corresponding DRBDSDEs, with obvious notations:
\begin{equation*}\begin{split}
Y_{t}^n =&\, \xi +\int_{t}^{T}f_s(W_s)ds - n \int_t^T (Y_s^n  - U_s)^{+} ds -\frac{1}{2}\int_{t}^{T}g_s*dW_s
+ \int_{t}^{T}h_s(W_s)\cdot\overleftarrow{dB}_s\\& -\sum_{i}\int_{t}^{T}Z^n_{i,s}dW_{s}^{i} + K_T^{+,n}-K_t^{+,n}
\end{split}\end{equation*}
and
\begin{equation*}\begin{split}
\bar{Y}_{t}^n =&\, \xi +\int_{t}^{T}\bar{f}_s(W_s)ds - n \int_t^T (\bar{Y}_s^n  - \bar{U}_s)^{+}ds -\frac{1}{2}\int_{t}^{T}\bar{g}_s*dW_s
+ \int_{t}^{T}\bar{h}_s(W_s)\cdot\overleftarrow{dB}_s \\&-\sum_{i}\int_{t}^{T}\bar{Z}^n_{i,s}dW_{s}^{i} + \bar{K}_T^{+,n}-\bar{K}_t^{+,n},
\end{split}\end{equation*}
where $\bar{Y}_t=y(t,W_t)$, $\bar{Z}_t=\nabla y(t,W_t)$ and $\bar{U}_t=\overline{o}(t,W_t)$.\\
Applying It\^o's formula to $\Phi(Y^n-\bar{Y}^n)$, for any $t\in[0,T]$, we have almost surely, 
\begin{equation*}\begin{split}
&\Phi(t,Y^n_t-\bar{Y}^n_t)=-\int_t^T\frac{\partial\Phi}{\partial s}(s,Y^n_s-\bar{Y}^n_s)ds-\int_t^T\Phi'(s,Y^n_s-\bar{Y}^n_s)(n(Y_s^n-U_s)^+-n(\bar{Y}_s^n-\bar{U}_s)^+)ds
\\&+\int_t^T\Phi'(s,Y^n_s-\bar{Y}^n_s)(f_s(W_s)-\bar{f}_s(W_s))ds-\frac{1}{2}\int_t^T\Phi'(s,Y^n_s-\bar{Y}^n_s)(g_s(W_s)-\bar{g}_s(W_s))*dW_s
\\&+\int_t^T\Phi'(s,Y^n_s-\bar{Y}^n_s)(h_s(W_s)-\bar{h}_s(W_s))\cdot\overleftarrow{dB}_s-\int_t^T\Phi'(s,Y^n_s-\bar{Y}^n_s)(Z^n_s-\bar{Z}^n_s)dW_s
\\&+\int_t^T\Phi'(s,Y^n_s-\bar{Y}^n_s)d(K_s^{+,n}-\bar{K}_s^{+,n})+\frac{1}{2}\int_t^T\Phi''(s,Y_s^n-\bar{Y}^n_s)|h_s(W_s)-\bar{h}_s(W_s)|^2ds
\\&-\frac{1}{2}\int_t^T\Phi''(s,Y_s^n-\bar{Y}^n_s)\langle g_s(W_s)-\bar{g}_s(W_s),Z^n_s-\bar{Z}^n_s\rangle ds-\frac{1}{2}\int_t^T\Phi''(s,Y_s^n-\bar{Y}^n_s)|Z^n_s-\bar{Z}^n_s|^2ds.
\end{split}\end{equation*}
Noting the following relation
\begin{equation*}\begin{split}
&\left|\int_t^T\Phi'(s,Y_s^n-\bar{Y}_s^n)d(K_s^{+,n}-\bar{K}_s^{+,n})-\int_t^T\Phi'(s,Y_s-\bar{Y}_s)d(K_s^{+}-\bar{K}_s^{+})\right|\\
\leq &\, \left|\int_t^T\Phi(s,Y_s^n-\bar{Y}_s^n)dK_s^{+,n}-\int_t^T\Phi'(s,Y_s-\bar{Y}_s)dK_s^+\right|
\\&\quad+\left|\int_t^T\Phi(s,Y_s^n-\bar{Y}_s^n)d\bar{K}_s^{+,n}-\int_t^T\Phi'(s,Y_s-\bar{Y}_s)d\bar{K}_s^+\right|\\
=&\,\left|\int_t^T(\Phi(s,Y_s^n-\bar{Y}_s^n)-\Phi'(s,Y_s-\bar{Y}_s))dK_s^{+,n}+\int_t^T\Phi'(s,Y_s-\bar{Y}_s)d(K_s^{+,n}-K_s^+)\right|\\
&\quad+\left|\int_t^T(\Phi(s,Y_s^n-\bar{Y}_s^n)-\Phi'(s,Y_s-\bar{Y}_s))d\bar{K}_s^{+,n}+\int_t^T\Phi'(s,Y_s-\bar{Y}_s)d(\bar{K}_s^{+,n}-\bar{K}_s^+)\right|.
\end{split}\end{equation*}
Then we can do a similar argument as in the proof of Theorem \ref{Itoformula}. 
Finally, due to the relation between $(u,\nu^+,\nu^-)$, $(y,\bar{\nu}^+,\bar{\nu}^-)$ and $(Y,Z,K^{+},K^-)$, $(\bar{Y},\bar{Z},\bar{K}^+,\bar{K}^-)$, we get the desired result. 
\end{proof}

\begin{remark}\label{remarkitoforpositivepart}
In the last two lemmas, we have proved an It\^o formula for a function $\Phi$ twice differentiable in space. Standard arguments based on an approximation of the function $x\longrightarrow (x^+ )^2$, see for example the proof of Lemma 7 in \cite{DMZ12}, permit to show  that formulas of Lemma \ref{itolowerobstacle} and Lemma \ref{itofordifference} still hold with $\Phi (x)=(x^+ )^2$ and in that case $\Phi'(x)=2x^+$ and $\Phi'' (x)=2{\bf 1}_{\{ x>0\}}$.
\end{remark}
{\red \section*{Acknowledgment}
The authors are very grateful to the editor and the anonymous referee for having pointed out several errors in the original version of this article and for their very valuable remarks and comments which have made it possible to improve it.}

\begin{center}
\begin{minipage}[t]{7cm}
Laurent DENIS \\
Le Mans Universit\'e\\
 Avenue Olivier Messiaen\\ F-72085 Le Mans Cedex 9, France \\
 e-mail: ldenis@univ-lemans.fr
\end{minipage}
\hfill
\begin{minipage}[t]{7cm}
 Anis MATOUSSI \\
Le Mans Universit\'e\\
 Avenue Olivier Messiaen\\ F-72085 Le Mans Cedex 9, France \\
email : anis.matoussi@univ-lemans.fr\\
and \\
CMAP,  Ecole Polytechnique, Palaiseau
\end{minipage}
\hfill \vspace*{0.8cm}
\begin{minipage}[t]{12cm}
Jing ZHANG \\
School of Mathematical Sciences \\
Fudan University,  Shanghai, China \\
Email: zhang\_jing@fudan.edu.cn
\end{minipage}
\end{center}

\end{document}